\documentclass[11pt]{amsart}
\usepackage{mathrsfs}
\usepackage{geometry}
\usepackage{microtype}
\usepackage{graphicx}
\usepackage{amssymb}
\usepackage{epstopdf}
\usepackage{amsmath,amscd}
\usepackage{amsthm}
\usepackage{enumerate}
\usepackage{url,verbatim}
\usepackage{subfig}
\usepackage{xcolor}

\usepackage[normalem]{ulem}

\RequirePackage[colorlinks,citecolor=blue,urlcolor=blue]{hyperref}

\theoremstyle{plain}
\newtheorem{theorem}{Theorem}
\newtheorem{lemma}[theorem]{Lemma}
\newtheorem{proposition}[theorem]{Proposition}
\newtheorem{corollary}[theorem]{Corollary}

\theoremstyle{definition}
\newtheorem{definition}[theorem]{Definition}
\newtheorem{example}[theorem]{Example}
\newtheorem{assumption}[theorem]{Assumption}

\theoremstyle{remark}
\newtheorem{remark}[theorem]{Remark}

\newcommand\ol{\overline}

\newcommand\EE{{\mathbb E}}

\newcommand\RR{{\mathbb R}}
\newcommand\ZZ{{\mathbb Z}}
\newcommand\NN{{\mathbb N}}

\newcommand\PP{{\mathbb P}}

\newcommand\GUE{{\mathbb {GUE}}}

\newcommand\si{\sigma}

\renewcommand\ell{l}

\newcommand\GT{\mathbb{G}\mathbb{T}}
\newcommand\CC{\mathbb{C}}
\newcommand\bm{\mathbf{m}}

\newcommand\YY{\mathbb{Y}}

\numberwithin{equation}{section}
\numberwithin{theorem}{section}
\numberwithin{figure}{section}

\title[Independent GUE  processes of perfect matchings on rail yard graphs]{Independent GUE  processes of perfect matchings on rail yard graphs}

\date{}

\author{Zhongyang Li}
\address{Department of Mathematics,
University of Connecticut,
Storrs, Connecticut 06269-3009, USA}
\email{zhongyang.li@uconn.edu}
\urladdr{\url{https://mathzhongyangli.wordpress.com}}

\begin{document}
\maketitle

\begin{abstract}We study perfect matchings on the rail-yard graphs in which the right boundary condition is given by the empty partition and the left boundary
can be divided into finitely many alternating line segments where all the vertices along
each line segment are either removed or remained.  When the edge
weights satisfy certain conditions, we show that the distributions of the locations of certain types of dimers near the right boundary converge to the spectra of independent GUE minor processes. The proof is based on new quantitative analysis of a formula to compute Schur functions at general points discovered in \cite{ZL18}.
\end{abstract}

\section{Introduction}

A Gaussian Unitary ensemble (GUE) is an $\infty\times \infty$ random Hermitian matrix $[H_{i,j}]_{i,j\in\NN}$ whose entries are given by independent random variables such that $H_{i,i}\sim \mathcal{N}(0,1)$ and for $i<j$, $H_{i,j}\sim \mathcal{N}\left(0,\frac{1}{2}\right)+\mathbf{i}\mathcal{N}\left(0,\frac{1}{2}\right)$ is a standard complex Gaussian random variable, whose real part and imaginary part are independent. 

A perfect matching, or a dimer cover of a graph is a subset of edges satisfying the constraint that each vertex is incident to exactly one edge. Perfect matchings are natural models in statistical mechanics, for example, dimer configurations on the hexagonal lattice model the molecule structure of graphite. See \cite{RK10,GV21} for more details.

Perfect matchings on 2-dimensional Euclidean lattice have been studied extensively. Among a number of interesting topics including phase transition (\cite{KOS06}), limit shape (\cite{CKP01,ko07,BL17,ZL20,kr20,ZL18,Li21,ZL20,LV21,ZL22,BD23,BB23}), conformal invariance (\cite{RK00,RK01,JD15,Li13,DJ19,BC21}),
the relations between distributions of random perfect matchings and spectra of GUE minor processes have attracted significant interest among probabilists for a long time. Here is an incomplete list. 
 Since the work of Johansson and Nordenstam
\cite{JN06}, it is known that the uniform random domino tilings of large size Aztec diamonds behave,
near the point of tangency of the inscribed arctic circle with the edge of the Aztec
diamond, like the spectra of the principal minors of a GUE-matrix. This was also
done in the work of Okounkov and Reshetikhin \cite{OR06} for plane partitions.
In the case
of the uniform perfect matching on a hexagon lattice, the result was proved in
\cite{VP15,JN14}. The convergence of locations of certain types of dimers near the tangent points to the spectra of GUE process in the scaling limit when the graph is a square-hexagon lattice with non-uniform weights was shown in \cite{BL17}. It was then proved in \cite{GA22} whenever a boundary of the domain has three adjacent straight segments inclined under 120 degrees to each other, the asymptotics of uniform random lozenge tilings near the middle segment is described by the GUE.

Multiple GUE minor processes arising from the same statistical mechanical model have gained significant attention recently, see for example \cite{AM13,NST23}.

Rail yard graphs are a general class of graphs introduced in \cite{bbccr} on which the random dimer coverings form Schur processes. Limit shapes and height fluctuations of dimer coverings on rail yard graphs were studied in \cite{Li21,LV21,ZL22}. The main result we proved in this paper, roughly speaking, is as follows

\begin{theorem}
 When the edge weights satisfy certain conditions, the distributions of the locations of certain types of dimers near the right boundary of a rail yard graph converge to the spectra of independent GUE minor processes in the scaling limit.
\end{theorem}
See Propositions \ref{p71} and \ref{p62} for precise statements.

The major differences of this paper and previous works lie in the following
\begin{itemize}
\item Previous works focus on GUE and dimer coverings on either hexagonal lattice or square grid; while our graph is the general rail yard graph.
\item Previous works focus on uniform dimer coverings and study  the distribution of certain types of dimers near one tangent points; when multiple tangent points are indeed discussed (see for example \cite{AM13}), these tangent points are formed due to the structure of the graph; while multiple tangent points in our paper are formed only because of the edge weights. Moreover, while \cite{AM13} considers 2 GUE processes, our results applies to $n$ GUE processes for arbitrary finite $n$.
\item We develop new techniques of asymptotic analysis of a formula to compute Schur functions at general points discovered in \cite{ZL18}. We also introduce new difference operators acting on schur polynomials, which, instead of acting on all the variables in a symmetric way, acts on part of the variables in a symmetric way, which give us the marginal distribution of certain parts of the random partition.
\end{itemize}

The organization of the paper is as follows. In Section~\ref{wryg}, we review the definitions of weighted rail yard graphs, dimer coverings and known results to compute the partition function of dimer coverings and Schur generating functions on such graphs.  We also introduce piecewise boundary conditions and the corresponding edge weights.
In Section~\ref{sec:gue}, we express the Schur generating function with given piecewise boundary condition as the product of Schur generating functions with boundary conditions given by one piece of the original boundary condition, when the edge weights satisfy certain conditions. We also construct subgraphs of rail yard graphs on which the Schur generating function of dimer coverings is given by factors of the original Schur generating function. The proof is based on new analysis of Schur generating functions inspired by sampling algorithms in~\cite{bbb14}. In Section~\ref{sect:md}, we introduce new difference operators, and prove that its action on the Schur generating function gives the marginal distribution of certain parts of the random partitions corresponding to random dimer coverings on rail yard graphs. In Section~\ref{sect:idp}, we prove that different parts of random partitions are asymptotically independent.
In Section~\ref{sect:sgue}, we prove that near the right boundary, the distribution of each part of the random partition converges to the spectra of GUE matrices when the size fo the graph goes to infinity.
In Section~\ref{sect:A}, we include a combinatorial formula to compute Schur functions proved in~\cite{ZL18}.

\section{Background}\label{wryg}

In this section, we review the definitions and properties of the Gaussian Unitary Ensemble (GUE), weighted rail yard graphs, dimer coverings and known results to compute the partition function of dimer coverings and Schur generating functions on such graphs. We then introduce the piecewise boundary condition and corresponding edge weights.

\subsection{Spectra of GUE matrices}
A matrix of the GUE of size $k$ is diagonalizable with real eigenvalues $\epsilon_1\geq
\epsilon_2\geq\ldots\geq\epsilon_k$, whose distribution $\PP_{\GUE_k}$ on $\RR^k$
has a density with respect to the Lebesgue measure on $\RR^k$ proportional to:
\begin{equation*}
\prod_{1\leq i<j\leq
k}(\epsilon_i-\epsilon_j)^2\exp\left(-\sum_{i=1}^{k}\epsilon_i^2\right),
\end{equation*}
See~\cite[Theorem~3.3.1]{MM04}.
Here is the main theorem we prove in this section.

The following lemma represents the Schur function as a matrix integral over the
unitary group, the so-called Harish--Chandra--Itzykson--Zuber integral:
\begin{lemma}[\cite{hc,iz}]
  \label{hciz}
  Let $\lambda\in \YY_N$ be a partition, and let $B$ be an
  $N\times N$ diagonal matrix given by
\begin{equation*}
B=\operatorname{diag}[\lambda_1+N-1,\ldots,\lambda_j+N-j,\ldots,\lambda_N+N-N]
\end{equation*}
Let $(a_1,a_2,\ldots,a_N)\in \CC^N$, and let $A$ be an $N\times N$ diagonal matrix given by 
\begin{equation*}
A=\operatorname{diag}[a_1,\ldots,a_N].
\end{equation*}
Then,
\begin{equation}
\frac{s_{\lambda}(e^{a_1},\ldots,e^{a_N})}{s_{\lambda}(1,\ldots,1)}=\prod_{1\leq i<j\leq N}\frac{a_i-a_j}{e^{a_i}-e^{a_j}}\int_{U(N)}e^{\mathrm{Tr}(U^*AUB)}dU,
\label{shc}
\end{equation}
where $dU$ is the Haar probability measure on the unitary group $U(N)$.
\end{lemma}

\begin{remark}
The standard definition of the Harish--Chandra–Itzykson–Zuber (HCIZ) integral on the right-hand side of \eqref{shc} allows arbitrary diagonal matrices \(A,B\in \mathbb{R}^{N\times N}\). In Lemma~\ref{hciz}, we specialize \(B\) by taking its diagonal entries to be positive integers, while we broaden \(A\) by permitting complex diagonal entries. The identity \eqref{shc} remains valid for \((a_1,\ldots,a_N)\in \mathbb{C}^N\) by analytic continuation. 

Moreover, if \(a_i = a_j\) for some \(i \neq j\), we take the limit \(a_i - a_j \to 0\) on both sides; then \eqref{shc} still holds.
\end{remark}

Recall the following lemma giving a necessary and sufficient condition for the distribution of eigenvalues of $k\times k$ GUE random matrix; see Lemma 6.2 of \cite{BL17}; see also \cite{GP15}.

\begin{lemma}
  \label{lgue}
  Let $(q_1,\dotsc,q_k)\in\RR^k$ be a random vector with distribution $\PP$ and
  $Q=\operatorname{diag}[q_1,\ldots,q_k]$ the $k\times k$ diagonal matrix
  obtained by putting the $q_i$ on the diagonal.

  Then $\PP$ is $\PP_{\GUE_k}$ if and only if for any matrix $P$,
\begin{equation*}
\EE \int_{U(k)}
\exp[\mathrm{Tr}(PUQU^*)]dU=
\exp\left(\frac{1}{2}\mathrm{Tr}P^2\right).
\end{equation*}
It is in fact enough to check the case when $P$ is diagonal, with real
coefficients.
\end{lemma}

\subsection{Weighted rail yard graphs and dimer covers}

Dimer coverings on weighted rail-yard graphs were first introduced and studied in \cite{bbccr}; we briefly review the basic definitions below.

Let $l,r\in\ZZ$ such that $l\leq r$. Let
\begin{eqnarray*}
[l..r]:=[l,r]\cap\ZZ,
\end{eqnarray*}
i.e., $[l..r]$ is the set of integers between $l$ and $r$. For a positive integer $m$, we use the notation $[m]$ to denote the set of all the integers from 1 to $m$, i.e.
\begin{eqnarray*}
[m]:=\{1,2,\ldots,m\}.
\end{eqnarray*}
Consider two binary sequences indexed by integers in $[l..r]$ 
\begin{itemize}
\item the $LR$ sequence $\underline{a}=\{a_l,a_{l+1},\ldots,a_r\}\in\{L,R\}^{[l..r]}$;
\item the sign sequence $\underline{b}=(b_l,b_{l+1},\ldots,b_l)\in\{+,-\}^{[l..r]}$.
\end{itemize}
The rail yard graph RYG$(l,r,\underline{a},\underline{b})$ with respect to integers $l$ and $r$, the $LR$ sequence $\underline{a}$ and the sign sequence $\underline{b}$, is the bipartite graph with vertex set $[2l-1..2r+1]\times \left\{\ZZ+\frac{1}{2}\right\}$. A vertex is called even (resp.\ odd) if its abscissa is an even (resp.\ odd) integer. Each even vertex $(2m,y)$, $m\in[l..r]$ is incident to 3 edges, two horizontal edges joining it to the odd vertices $(2m-1,y)$ and $(2m+1,y)$ and one diagonal edge joining it to
\begin{itemize}
\item the odd vertex $(2m-1,y+1)$ if $(a_m,b_m)=(L,+)$;
\item the odd vertex $(2m-1,y-1)$ if $(a_m,b_m)=(L,-)$;
\item the odd vertex $(2m+1,y+1)$ if $(a_m,b_m)=(R,+)$;
\item the odd vertex $(2m+1,y-1)$ if $(a_m,b_m)=(R,-)$.
\end{itemize} 

See Figure \ref{fig:rye} for an example of a rail yard graph.
\begin{figure}
\includegraphics[width=.8\textwidth]{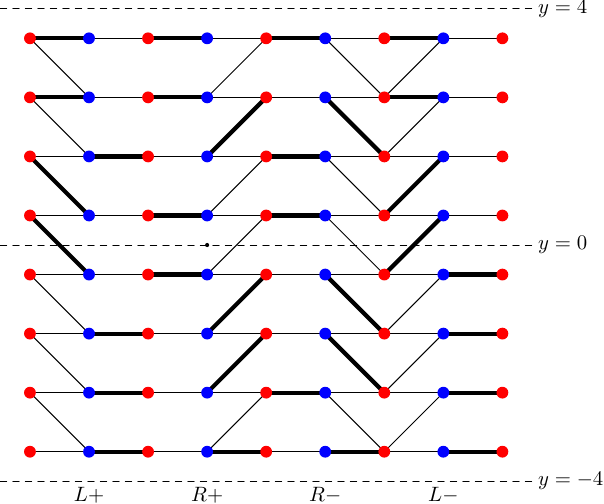}
\caption{A rail yard graph with LR sequence $\underline{a}=\{L,R,R,L\}$and sign sequence $\underline{b}=\{+,+,-,-\}$. Odd vertices are represented by red points, and even vertices are represented by blue points. Dark lines represent a pure dimer covering. Assume that above the horizontal line $y=4$, only horizontal edges with an odd vertex on the left are present in the dimer configuration; and below the horizontal line $y=-4$, only horizontal edges with an even vertex on the left are present in the dimer configuration. $l=0$ and $r=3$; red vertices have abscissas $x=-1,x=1,x=3,x=5,x=7$; blue vertices have abscissas $x=0,x=2,x=4,x=6$.Parameters atre given by $(a_0,b_0)=(L,+)$, $(a_1,b_1)=(R,+)$,$(a_2,b_2)=(R,-)$ and $(a_3,b_3)=(L,-)$.}\label{fig:rye}
\end{figure}

The left boundary (resp.\ right boundary) of RYG$ (l,r,\underline{a},\underline{b})$ consists of all odd vertices with abscissa $2l-1$ (resp.\ $2r+1$). Vertices which do not belong to the boundaries are called inner. A face of RYG$(l,r,\underline{a},\underline{b})$ is called an inner face if it contains only inner vertices.

We assign edge weights to a rail yard graph RYG$(l,r,\underline{a},\underline{b})$ as follows:
\begin{itemize}
    \item all the horizontal edges have weight 1; and
    \item each diagonal edge adjacent to a vertex with abscissa $2i$ has weight $x_i$.
\end{itemize}

\begin{definition}
A dimer covering is a subset of edges of RYG$(l,r,\underline{a},\underline{b})$ such that
\begin{enumerate}
\item each inner vertex of RYG$(l,r,\underline{a},\underline{b})$ is incident to exactly one edge in the subset;
\item each left boundary vertex or right boundary vertex is incident to at most one edge in the subset;
\item only a finite number of diagonal edges are present in the subset.
\end{enumerate}
\end{definition}

See Figure \ref{fig:rye} for an example of a dimer covering on a rail yard graph.

\subsection{Partitions}

Let $N$ be a non-negative integer. A length-$N$ partition is a non-increasing sequence $\lambda=(\lambda_i)_{i\geq 1}^{N}$ of non-negative integers. Let $\mathbb{Y}_N$ be the set of all the length-$N$ partitions. The size of a partition $\lambda\in \YY_N$ is defined by
\begin{eqnarray*}
|\lambda|=\sum_{i=1}^{N}\lambda_i.
\end{eqnarray*}
We denote the length of a partition by $l(\lambda)$. In particular, if $\lambda\in \YY_N$, then $l(\lambda)=N$. 

Given a partition $\lambda=(\lambda_1,\ldots,\lambda_\ell)$, its
\emph{Young diagram} $Y(\lambda)$ is the left-justified array of unit boxes with
$\lambda_i$ boxes in row $i$ (rows indexed from top to bottom).
Equivalently, if $D$ is a left-justified array of boxes with row-lengths
$r_1\ge r_2\ge\cdots\ge r_\ell$, then the associated partition is
$\lambda=(r_1,\ldots,r_\ell)$. See Figure \ref{fig:yd} for the Young diagram corresponding to the partition $\lambda=(5,3,3,1)$.

\begin{figure}
\includegraphics[width=.25\textwidth]{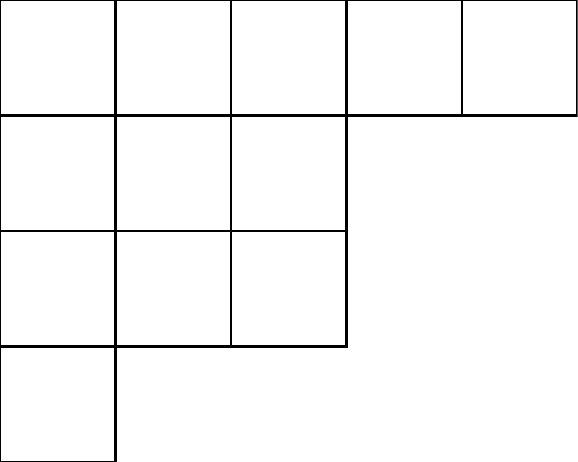}
\caption{Young diagram corresponding to the partition $\lambda=(5,3,3,1)$.}\label{fig:yd}
\end{figure}

Two partitions $\lambda\in \YY_N$ and $\mu\in\YY_{N-1}$ are called interlaced, and written by $\lambda\succ\mu$ or $\mu\prec \lambda$ if
\begin{eqnarray*}
\lambda_1\geq \mu_1\geq\lambda_2\geq \mu_2\geq \lambda_3\geq \ldots\geq \mu_{N-1}\geq \lambda_N.
\end{eqnarray*}

A length-$\infty$ partition is a non-increasing sequence $\lambda=(\lambda_i)_{i\geq 1}^{\infty}$ of non-negative integers which vanishes eventually. Each length-$N$ partition can be naturally extended to an length-$\infty$ partition by adding infinitely many $0$'s at the end of the sequence.

When representing partitions by Young diagrams, $\mu\prec \lambda$ means $\lambda/\mu$ is a horizontal strip. The conjugate partition $\lambda'$ of $\lambda$ is a partition whose Young diagram $Y_{\lambda'}$ is the image of the Young diagram $Y_{\lambda}$ of $\lambda$ by the reflection along the main diagonal. More precisely
\begin{eqnarray*}
\lambda_i':=\left|\{j\geq 1: \lambda_j\geq i\}\right|,\qquad \forall i\geq 1.
\end{eqnarray*}

The skew Schur functions are defined in Section I.5 of \cite{IGM15}.

\begin{definition}\label{dss}Let $\lambda$, $\mu$ be partitions of finite length. Define the skew Schur functions as follows
\begin{eqnarray*}
&&s_{\lambda/\mu}=\det\left(h_{\lambda_i-\mu_j-i+j}\right)_{i,j=1}^{l(\lambda)}\\
\end{eqnarray*} 
Here for each $r\geq 0$, $h_r$ is the $r$th complete symmetric function defined by the sum of all monomials of total degree $r$ in the variables $x_1,x_2,\ldots$. More precisely,
\begin{eqnarray*}
h_r=\sum_{1\leq i_1\leq i_2\leq \ldots\leq i_r} x_{i_1}x_{i_2}\cdots x_{i_r}
\end{eqnarray*}
If $r<0$, $h_r=0$.

Define the Schur function as follows
\begin{eqnarray*}
s_{\lambda}=s_{\lambda/\emptyset}.
\end{eqnarray*}
\end{definition}

For a dimer covering $M$ of RYG$(l,r,\underline{a},\underline{b})$, we associate a particle-hole configuration to each odd vertex of RYG$(l,r,\underline{a},\underline{b})$ as follows: let $m\in[l..(r+1)]$ and $k\in\ZZ$: if the odd endpoint $\left(2m-1,k+\frac{1}{2}\right)$ is incident to a present edge in $M$ on its right (resp.\ left), then associate a hole (resp.\ particle) to the odd endpoint $\left(2m-1,k+\frac{1}{2}\right)$. When $M$ is a pure dimer covering, it is not hard to check that there exists $N>0$, such that when $y>N$, only holes exist and when $y<-N$, only particles exist.

We associate a partition $\lambda^{(M,m)}$ to the column indexed by $m$ of particle-hole configurations, which corresponds to a pure dimer covering $M$ adjacent to odd vertices with abscissa $(2m-1)$ as follows. Assume
\begin{eqnarray*}
\lambda^{(M,m)}=(\lambda^{(M,m)}_1,\lambda^{(M,m)}_2,\ldots),
\end{eqnarray*}
Then for $i\geq 1$, $\lambda^{(M,m)}_i$ is the total number of holes in $M$ along the vertical line $x=2m-1$ below the $i$th highest particles. Let $l(\lambda^{(M,m)})$ be the total number of nonzero parts in the partition $\lambda^{(M,m)}$.

\begin{figure}
\includegraphics[width=.6\textwidth]{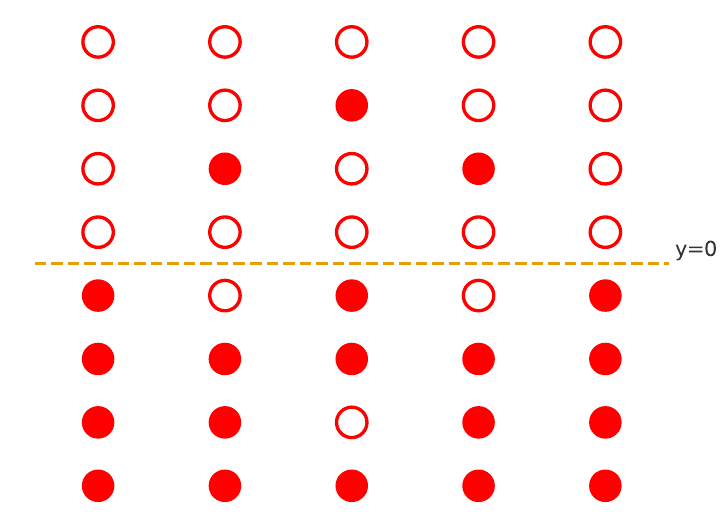}
\caption{Particle-hole configuration corresponding to the dimer covering in Figure \ref{fig:rye} at red vertices. Particles are represented by red dots, while holes are represented by red circles. Each particle/hole represents the configuration of an odd vertex (represented by red dots)  in the same location of Figure \ref{fig:rye}. The corresponding sequence of partitions (from the left to the right) is given by $\emptyset\prec(2,0,\ldots)\prec' (3,1,1,0\ldots)\succ'(2,0,\ldots)\succ \emptyset$.}\label{fig:rye}
\end{figure}

We define the charge $c^{(M,m)}$ on column $(2m-1)$ for the configuration $M$ as follows:
\begin{eqnarray}
c^{(M,m)}&=&\mathrm{number\ of\ particles\ on\ column\ }(2m-1)\ \mathrm{in\ the\ upper\ half\ plane}\label{dcg}\\
&&-\mathrm{number\ of\ holes\ on\ column\ }(2m-1)\ \mathrm{in\ the\ lower\ half\ plane}\notag
\end{eqnarray}

Recall the following lemma proved in \cite{ZL22}.

\begin{lemma}Let $M$ be a dimer covering of RYG$(l,r,\underline{a},\underline{b})$. Then for each $m\in [l..r]$, the charge $c^{(M,m)}$ is a constant independent of $m$.
\end{lemma}

\begin{proof}See Lemma 3.2 of \cite{ZL22}; see equation (26) in Proposition 8 of \cite{bbccr}.
\end{proof}

The weight of a dimer covering $M$ of $RYG(l,r,\underline{a},\underline{b})$ is defined as follows
\begin{eqnarray*}
w(M):=\prod_{i=l}^{r}x_i^{d_i(M)},
\end{eqnarray*}
where $d_i(M)$ is the total number of present diagonal edges of $M$ incident to an even vertex with abscissa $2i$. 

Let $\lambda^{(l)},\lambda^{(r+1)}$ be two partitions. The partition function $Z_{\lambda^{(l)},\lambda^{(r+1)}}(G,\underline{x})$ of dimer coverings on RYG$(l,r,\underline{a},\underline{b})$ whose configurations on the left (resp.\ right) boundary correspond to partition $\lambda^{(l)}$ (resp.\ $\lambda^{(r+1)}$) is the sum of weights of all such dimer coverings on the graph. Given the left and right boundary conditions $\lambda^{(l)}$ and $\lambda^{(r+1)}$, respectively, the probability of a dimer covering $M$ is then defined by
\begin{eqnarray}
\mathrm{Pr}(M|\lambda^{(l)},\lambda^{(r+1)}):=\frac{w(M)}{Z_{\lambda^{(l)},\lambda^{(r+1)}}(G,\underline{x})}.\label{ppd}
\end{eqnarray}
Note that pure dimer coverings have left and right boundary conditions given by
\begin{eqnarray}
\lambda^{(l)}=\lambda^{(r+1)}=\emptyset;\label{pbc}
\end{eqnarray}
respectively.

The bosonic Fock space $\mathcal{B}$ is the infinite dimensional Hilbert space spanned by the orthonormal basis vectors $|\lambda\rangle$, where $\lambda$ runs over all the partitions. Let $\langle \lambda |$ denote the dual basis vector. Let $x$ be a formal or a complex variable. Introduce the operators $\Gamma_{L+}(x)$, $\Gamma_{L-}(x)$, $\Gamma_{R+}(x)$, $\Gamma_{R-}(x)$ from $\mathcal{B}$  to $\mathcal{B}$ as follows
\begin{eqnarray*}
\Gamma_{L+}(x)|\lambda\rangle=\sum_{\mu\prec \lambda}x^{|\lambda|-|\mu|}|\mu\rangle;\qquad \Gamma_{R+}(x)|\lambda\rangle=\sum_{\mu'\prec \lambda'}x^{|\lambda|-|\mu|}|\mu\rangle;\\
\Gamma_{L-}(x)|\lambda\rangle=\sum_{\mu\succ \lambda}x^{|\mu|-|\lambda|}|\mu\rangle;\qquad \Gamma_{R-}(x)|\lambda\rangle=\sum_{\mu'\succ \lambda'}x^{|\mu|-|\lambda|}|\mu\rangle;
\end{eqnarray*}
Such operators were used in \cite{oko01} to study random partitions.

\begin{lemma}\label{l12}Let $a_1,a_2\in \{L,R\}$. We have the following commutation relations for the operators $\Gamma_{a_1,\pm}$, $\Gamma_{a_2,\pm}$.
\begin{eqnarray*}
\Gamma_{a_1,+}(x_1)\Gamma_{a_2,-}(x_2)=\begin{cases}\frac{\Gamma_{a_2,-}(x_2)\Gamma_{a_1,+}(x_1)}{1-x_1x_2}&\mathrm{if}\ a_1=a_2\\(1+x_1x_2)\Gamma_{a_2,-}(x_2)\Gamma_{a_1,+}(x_1)&\mathrm{if}\ a_1\neq a_2.\end{cases}.
\end{eqnarray*}
Moreover,
\begin{eqnarray*}
\Gamma_{a_1,b}(x_1)\Gamma_{a_2,b}(x_2)=\Gamma_{a_2,b}(x_2)\Gamma_{a_1,b}(x_1);
\end{eqnarray*}
for all $a_1,a_2\in\{L,R\}$ and $b\in\{+,-\}$.
\end{lemma}

\begin{proof}See Proposition 7 of \cite{bbccr}; see also \cite{you10,bbb14}.
\end{proof}

Given the definitions of the operators $\Gamma_{a,b}(x)$ with $a\in\{L,R\}$, $b\in\{+,-\}$, it is straightforward to check the following lemma.

\begin{lemma}\label{l13}The partition function of dimer coverings on a rail yard graph $G=RYG(l,r,\underline{a},\underline{b})$ with left and right boundary conditions given by $\lambda^{(l)},\lambda^{(r+1)}$, respectively,  is 
\begin{eqnarray}
Z_{\lambda^{(l)},\lambda^{(r+1)}}(G;\mathbf{x})=\langle\lambda^{(l)}| \Gamma_{a_lb_l}(x_l)\Gamma_{a_{l+1}b_{l+1}}(x_{l+1})\cdots \Gamma_{a_rb_r}(x_r)|\lambda^{(r+1)} \rangle;\label{bf}
\end{eqnarray}
where
\begin{align}
\mathbf{x}:=(x_{l},\ldots,x_{r})
\end{align}
denote the edge weights.
\end{lemma}

\begin{assumption}\label{ap14}For any $i,j\in[l..r]$, $i<j$, $a_i=a_j$ and $b_i=+$, $b_j=-$ we have
\begin{eqnarray*}
x_ix_j<1.
\end{eqnarray*}
\end{assumption}

\begin{corollary}\label{lc13}Suppose Assumption \ref{ap14} holds.
Assume for all
$i\in[l..r]$,
\begin{eqnarray}
(a_i,b_i)\neq (R,-).\label{c151}
\end{eqnarray}
Let $g_l$ be the number of nonzero entries in $\lambda^{(l)}$, then
\begin{eqnarray}
g_l\leq \left|\{i\in[l..r]:a_i=L,b_i=-\right\}|.\label{glu}
\end{eqnarray}
The partition function of dimer coverings on a rail yard graph $G=RYG(l,r,\underline{a},\underline{b})$ with left and right boundary conditions given by $\lambda^{(l)},\emptyset$, respectively,  can be computed as follows:
\begin{eqnarray}
Z_{\lambda^{(l)},\emptyset}(G;\mathbf{x})=s_{\lambda^{(l)}}\left(\mathbf{x}^{(L,-)}\right)\prod_{l\leq i<j\leq r;b_i=+,b_j=-}z_{i,j}\label{fp}
\end{eqnarray}
where
\begin{eqnarray}
z_{ij}=\begin{cases}1+x_ix_j&\mathrm{if}\ a_i\neq a_j\\\frac{1}{1-x_ix_j}&\mathrm{if}\ a_i=a_j\end{cases}\label{dzij}
\end{eqnarray}
and
\begin{eqnarray*}
\mathbf{x}^{(L,-)}:=\{x_i: (a_i,b_i)=(L,-)\}.
\end{eqnarray*}
\end{corollary}

\begin{proof}See Corollary 2.5 of \cite{Li21}.
\end{proof}

\begin{remark}The condition~\ref{c151} restricts the four a priori possible parameter pairs
attached to even columns—\((L,-),\,(R,+),\,(R,-),\,(L,+)\)—to the three
admissible pairs \((L,-),\,(R,+),\,(L,+)\). In particular, if \(b_i=-\) then
necessarily \(a_i=L\).

The condition~\ref{c151} ensures that the partition function can be written as
the product of a Schur polynomial and a rational factor; see~\eqref{fp}.
Consequently, the scaling limit of dimer coverings can be studied by analyzing
the asymptotics of Schur polynomials. In the absence of~\ref{c151}, the
factorization~\eqref{fp} no longer holds.
\end{remark}

\subsection{Schur generating functions}
\begin{definition}
  \label{df21}
  Let $\mathbf{X}=(x_1,\ldots,x_N)\in\CC^{N}$. Let $\rho$ be a
  probability measure on $\YY_N$. The \emph{Schur generating function}
  $\mathcal{S}_{\rho,\mathbf{X}}(u_1,\ldots,u_N)$ with respect to parameters
  $\mathbf{X}$ is the symmetric Laurent series in $(u_1,\ldots,u_N)$ given by
  \begin{equation*}
    \mathcal{S}_{\rho,\mathbf{X}}(u_1,\ldots,u_N)=
    \sum_{\lambda\in \YY} \rho(\lambda)
    \frac{s_{\lambda}(u_1,\ldots,u_N)}{s_{\lambda}(x_1,\ldots,x_N)}.
  \end{equation*}
\end{definition}

\begin{lemma}
  \label{lm212}
  Let $t\in[l..r]$. Assume (\ref{c151}) holds. Let
\begin{eqnarray*}
\mathbf{x}^{(L,-,> t)}&=&\{x_i:i\in[t+1...r],a_i=L,b_i=-\}.\\
\mathbf{x}^{(L,-,\leq t)}&=&\{x_i:i\in[l...t],a_i=L,b_i=-\}
\end{eqnarray*}
Let
\begin{eqnarray*}
\mathbf{u}^{(L,-,> t)}&=&\{u_i:i\in[t+1...r],a_i=L,b_i=-\}.
\end{eqnarray*}
Let $\rho^t$ be the probability measure for the partitions corresponding to dimer configurations adjacent to the column of odd vertices labeled by $2t-1$ in $RYG(l,r,\underline{a},\underline{b})$, conditional on the left and right boundary condition $\lambda^{(l)}$ and $\emptyset$, respectively.

For $i\in[l..r]$, let
  \begin{eqnarray*}
  w_i=\begin{cases}u_i&\mathrm{if}\ i\in[t+1..r],a_i=L,b_i=-\\ x_i&\mathrm{otherwise}\end{cases}
  \end{eqnarray*}
  and 
  \begin{eqnarray*}
\xi_{ij}=\begin{cases}1+w_iw_j&\mathrm{if}\ a_i\neq a_j\\\frac{1}{1-w_iw_j}&\mathrm{if}\ a_i=a_j\end{cases}.
\end{eqnarray*}

  Then the generating Schur function $\mathcal{S}_{\rho^t,\mathbf{x}^{(L,-,>t)}}$ is given by:
\begin{equation}
\mathcal{S}_{\rho^t,\mathbf{x}^{(L,-,> t)}}(\mathbf{u}^{(L,-,> t)})=
\frac{s_{\lambda^{(l)}}\left(\mathbf{u}^{(L,-,> t)},\mathbf{x}^{(L,-,\leq t)}\right)}{s_{\lambda^{(l)}}(\mathbf{x}^{(L,-)})}
\prod_{i\in[l..t],b_i=+;}
\prod_{j\in[t+1..r],a_j=L,b_j=-.}\frac{\xi_{ij}}{z_{ij}}.\label{fsg}
\end{equation}
\end{lemma}

\begin{proof}See Lemma 2.8 of \cite{Li21}.
\end{proof}

\begin{remark}The formula \eqref{fsg} for the Schur generating function follows from the Schur
branching rule; see the proof of Lemma~2.8 in \cite{Li21}. Note that \emph{both
sides} of \eqref{fsg} depend on \(\mathbf{x}^{(L,-,\le t)}\). On the
\emph{right–hand side}, the Schur polynomial
\(s_{\lambda^{(l)}}\!\big(\mathbf{x}^{(L,-)}\big)\) depends on
\(\mathbf{x}^{(L,-,\le t)}\). By Definition~\ref{df21}, the Schur generating
function is the expectation (under a probability measure) of a ratio of two
Schur polynomials. Consequently, the \emph{left–hand side} of \eqref{fsg} also
depends on \(\mathbf{x}^{(L,-,\le t)}\), because the measure \(\rho^t\) itself
is a function of \(\mathbf{x}^{(L,-,\le t)}\).
\end{remark}

\subsection{Piecewise Boundary Conditions}\label{pwbd}

\begin{assumption}\label{ap41}Let $\epsilon>0$ be a small positive parameter. 
\begin{enumerate}
\item Let $l^{(\epsilon)}<r^{(\epsilon)}$ be the  integers representing the left and right boundary of the rail yard graph depending on $\epsilon$ such that
\begin{eqnarray*}
\lim_{\epsilon\rightarrow 0}\epsilon l^{(\epsilon)}=l^{(0)}<r^{(0)}=\lim_{\epsilon\rightarrow 0}\epsilon r^{(\epsilon)}
\end{eqnarray*}
so that the scaling limit of the sequence of rail yard graphs $\{\epsilon RYG(l^{(\epsilon)},r^{(\epsilon)},\underline{a}^{(\epsilon)},\underline{b}^{(\epsilon)})\}_{\epsilon>0}$, as $\epsilon\rightarrow 0$, has left boundary given by $x=2l^{(0)}$ and right boundary given by $x=2r^{(0)}$.
\item Assume for each $\epsilon>0$, the weight of each diagonal edge incident to an even vertex with abscissa $2i$, for $i\in[l^{(\epsilon)}..r^{(\epsilon)}]$ is $x_{i,\epsilon}$.
\item Assume for all $i^{(\epsilon)}\in[l^{(\epsilon)}..r^{(\epsilon)}]$, (\ref{c151}) holds.
\end{enumerate}$\hfill\Box$
\end{assumption}

Suppose that Assumption \ref{ap41} holds. For each $\epsilon>0$, define
\begin{eqnarray*}
N_{L,-}^{(\epsilon)}:=\left|\left\{j\in\left[l^{(\epsilon)}..r^{(\epsilon)}\right]:a_j^{(\epsilon)}=L,\ b_j^{(\epsilon)}=-\right\}\right|.
\end{eqnarray*}
In other words, $N_{L,-}^{(\epsilon)}$ is the number of even columns labeled $(L,-)$.

Let $\pi$ be the bijection from $\left[N_{L,-}^{(\epsilon)}\right]$ to $\{j\in[l^{(\epsilon)}..r^{(\epsilon)}]:a_j^{(\epsilon)}=L,b_j^{(\epsilon)}=-\}$, such that
\begin{itemize}
    \item for any $i,j\in\left[N_{L,-}^{(\epsilon)}\right]$ and $i<j$, $\pi(i)<\pi(j)$.
\end{itemize}

Let
\begin{align}
N^{(\epsilon)}:=r^{(\epsilon)}-l^{(\epsilon)}+1;\label{dne}
\end{align}
for all $\epsilon>0$; i.e.~$N^{(\epsilon)}$ is the total number of even columns in the rail yard graph.

Let $\Sigma_{L,-,\epsilon}$ be the group consisting of all the bijections (permutations) from $\{j\in[l^{(\epsilon)}..r^{(\epsilon)}]:a_j^{(\epsilon)}=L,b_j^{(\epsilon)}=-\}$ to itself.

Let $\Sigma_{L,-,\epsilon}^X$ be a subgroup of $\Sigma_{L,-,\epsilon}$ preserving the values of the diagonal edge weights $x_{i,\epsilon}$'s, i.e.
\begin{eqnarray*}
\Sigma_{L,-,\epsilon}^X:=\{\si\in \Sigma_{L,-,\epsilon}:x_{\si(i),\epsilon}=x_{i,\epsilon},\ \forall\ i\in[l^{(\epsilon)}..  r^{(\epsilon)}],\ \mathrm{s.t.}\ a_i^{(\epsilon)}=L,b_i^{(\epsilon)}=-\}.
\end{eqnarray*}

Note that $\pi\Sigma_{L,-,\epsilon}$, defined by
\begin{eqnarray}
\pi\Sigma_{L,-,\epsilon}:=\{\pi\xi:\xi\in \Sigma_{L,-,\epsilon}\},\label{dps}
\end{eqnarray}
consists of all the bijections from $\left[N_{L,-}^{(\epsilon)}\right]$ to $\{j\in[l^{(\epsilon)}..r^{(\epsilon)}]:a_j^{(\epsilon)}=L,b_j^{(\epsilon)}=-\}$.
Let $\sigma_0\in\pi\Sigma_{L,-,\epsilon}$ such that
\begin{eqnarray}
x_{\si_0(1),\epsilon}\geq x_{\si_0(2),\epsilon}\geq\ldots\geq x_{\si_0\left(N_{L,-}^{(\epsilon)}\right),\epsilon}.\label{sz}
\end{eqnarray}

\begin{example}\label{ex213}Let $[l^{(\epsilon)}..r^{(\epsilon)}]=\{5,6,7,8\}$. Assume 
\begin{itemize}
\item $(a_5^{(\epsilon)},b_5^{(\epsilon)}) 
=(a_6^{(\epsilon)},b_6^{(\epsilon)})=(a_8^{(\epsilon)},b_8^{(\epsilon)})=(L,-)$, and $(a_7^{(\epsilon)},b_7^{(\epsilon)})=(R,+)$. 
\item $x_5^{(\epsilon)}=x_6^{(\epsilon)}=x_7^{(\epsilon)}=1$ and $x_8^{(\epsilon)}=\frac{1}{2}$
\end{itemize}
Then $N^{(\epsilon)}_{L,-}=3$, $N^{(\epsilon)}=4$. 
$\pi$ is the bijection from $\{1,2,3\}$ to $\{5,6,8\}$ satisfying
\begin{align*}
\pi(1)=5;\ \pi(2)=6;\ \pi(3)=8.
\end{align*}

\begin{small}
\begin{align*}
\Sigma_{L,-,\epsilon}=\left\{\left(\begin{array}{ccc}5&6&8\\5&6&8\end{array}\right),\left(\begin{array}{ccc}5&6&8\\5&8&6\end{array}\right),\left(\begin{array}{ccc}5&6&8\\6&5&8\end{array}\right),\left(\begin{array}{ccc}5&6&8\\6&8&5\end{array}\right),\left(\begin{array}{ccc}5&6&8\\8&5&6\end{array}\right),\left(\begin{array}{ccc}5&6&8\\8&6&5\end{array}\right)\right\}.
\end{align*}
\end{small}
\begin{small}
\begin{align*}
\Sigma_{L,-,\epsilon}^{X}=\left\{\left(\begin{array}{ccc}5&6&8\\5&6&8\end{array}\right),\left(\begin{array}{ccc}5&6&8\\6&5&8\end{array}\right)\right\}.
\end{align*}
\end{small}
\begin{small}
\begin{align*}
\pi\Sigma_{L,-,\epsilon}=\left\{\left(\begin{array}{ccc}1&2&3\\5&6&8\end{array}\right),\left(\begin{array}{ccc}1&2&3\\5&8&6\end{array}\right),\left(\begin{array}{ccc}1&2&3\\6&5&8\end{array}\right),\left(\begin{array}{ccc}1&2&3\\6&8&5\end{array}\right),\left(\begin{array}{ccc}1&2&3\\8&5&6\end{array}\right),\left(\begin{array}{ccc}1&2&3\\8&6&5\end{array}\right)\right\}.
\end{align*}
\end{small}
\begin{align*}
\si_0=\left(\begin{array}{ccc}1&2&3\\5&6&8\end{array}\right)\ \mathrm{or}\si_0=\left(\begin{array}{ccc}1&2&3\\6&5&8\end{array}\right)
\end{align*}
\end{example}

\begin{assumption}\label{ap428}Suppose Assumption \ref{ap41} holds. Let $s\in\left[N_{L,-}^{(\epsilon)}\right]$ such that $s$ is independent of $\epsilon$ as $\epsilon\rightarrow 0$.
Assume there exist positive integers $K_1^{(\epsilon)},K_2^{(\epsilon)},\ldots K_s^{(\epsilon)}$, such that 
\begin{enumerate}
\item let $n>0$ be the total number of distinct elements in $\{x_{\si_0(1),\epsilon},\ldots, x_{\si_0(N_{L,-}^{(\epsilon)}),\epsilon}\}$; then $n$ is finite, fixed and independent of $\epsilon$.
\item $\sum_{t=1}^s K_t^{(\epsilon)}=N_{L,-}^{(\epsilon)}$;
\item Assume the partition corresponding to the left boundary condition is $\lambda^{(l,\epsilon)}$ is given by
\begin{eqnarray*}
\lambda^{(l,\epsilon)}=\left(\lambda_1,\ldots,\lambda_{N_{L,-}^{(\epsilon)}}\right).
\end{eqnarray*}
such that
\begin{eqnarray}
\mu_1^{(\epsilon)}>\ldots>\mu_s^{(\epsilon)}\label{mi}
\end{eqnarray}
are all the distinct elements in $\left\{\lambda_1,\lambda_2,\ldots,\lambda_{N_{L,-}^{(\epsilon)}}\right\}$.
\item
\begin{eqnarray*}
&&\lambda_1=\lambda_2=\ldots=\lambda_{K_s^{(\epsilon)}}=\mu_1^{(\epsilon)};\\
&&\lambda_{K_s^{(\epsilon)}+1}=\lambda_{K_s^{(\epsilon)}+2}=\ldots=\lambda_{K_s^{(\epsilon)}+K_{s-1}^{(\epsilon)}}=\mu_2^{(\epsilon)};\\
&&\ldots\\
&&\lambda_{1+\sum_{t=2}^{s}K_t^{(\epsilon)}}=\lambda_{2+\sum_{t=2}^{s}K_t^{(\epsilon)}}=\ldots=\lambda_{\sum_{t=1}^{s}K_t^{(\epsilon)}}=\mu_s^{(\epsilon)};
\end{eqnarray*}
 \item For $i\in[n]$, let 
\begin{eqnarray}
\label{ji}J_{i}=\{t\in[s]:\exists q\in\left[N_{L,-}^{(\epsilon)}\right],\ \mathrm{s.t.}\ x_{\sigma_0(q)}\ \mathrm{is\ the}\ i\mathrm{th\ largest\ elements\ in}\\\{x_{\si_0(1),\epsilon},\ldots,x_{\si_0(N_{L,-}^{(\epsilon)}),\epsilon}\},
\mathrm{and}\ \lambda_q=\mu_t^{(\epsilon)}\}\notag
\end{eqnarray}
such that
\begin{enumerate}
\item For any $i_1,i_2\in[n]$ with $i_1<i_2$, 
if $\ell\in J_{i_1}$, and $t\in J_{i_2}$, then $\ell<t$.
\item For any $p,q$ satisfying $1\leq p\leq s$ and $1\leq q\leq s$, and $q>p$
\begin{eqnarray*}
C_1N_{L,-}^{(\epsilon)} \leq \mu_p^{(\epsilon)}-\mu_q^{(\epsilon)}\leq C_2N_{L,-}^{(\epsilon)}
\end{eqnarray*}
where $C_1$, $C_2$ are sufficiently large constants independent of $\epsilon$ as $\epsilon\rightarrow 0$.
\end{enumerate}
\end{enumerate}$\hfill\Box$
\end{assumption}

\begin{example}In Example \ref{ex213}, $n=2$ since the total number of distinct numbers in 
$\{x_5^{(\epsilon)},x_6^{(\epsilon)},x_8^{(\epsilon)}\}=\{1,1,\frac{1}{2}\}$ is 2. 
\begin{enumerate}
\item If the partition corresponding to the left boundary condition is given by
\begin{align*}
\lambda^{(l,\epsilon)}=(3,2,1)
\end{align*}
then $s=3$, 
\begin{align*}
\mu_1^{(\epsilon)}=3;\ \mu_2^{(\epsilon)}=2;\ \mu_3^{(\epsilon)}=1,
\end{align*}
\begin{align*}
K_3^{(\epsilon)}=K_2^{(\epsilon)}=K_1^{(\epsilon)}=1.
\end{align*}
\begin{align*}
J_1=\{1,2\};\ J_2=\{3\}.
\end{align*}
One can verify that Assumption \ref{ap428}(5a) holds.
\item If the partition corresponding to the left boundary condition is given by
\begin{align*}
\lambda^{(l,\epsilon)}=(3,2,2)
\end{align*}
then $s=2$, 
\begin{align*}
\mu_1^{(\epsilon)}=3;\ \mu_2^{(\epsilon)}=2.
\end{align*}
\begin{align*}
K_2^{(\epsilon)}=1;\ K_1^{(\epsilon)}=2.
\end{align*}
\begin{align*}
J_1=\{1,2\};\ J_2=\{2\}.
\end{align*}
One can verify that Assumption \ref{ap428}(5a) does not hold.
\end{enumerate}
\end{example}

\begin{assumption}\label{ap32}Suppose that Assumption \ref{ap41} holds. For $i\in[l^{(\epsilon)}..r^{(\epsilon)}]$, suppose the weights $x_i$ change with respect to $\epsilon$, and denote it by $x_{i,\epsilon}$. Let $\sigma_0$ be defined as in (\ref{sz}). More precisely, we order elements in the set
\begin{eqnarray}
\{x_{i,\epsilon}\}_{i\in[l^{(\epsilon)},r^{\epsilon}],a_{i,p}=L,b_{i,p}=-}\label{wl-}
\end{eqnarray}
by (\ref{sz}).
Assume 
\begin{eqnarray*}
x_{\sigma_0(1),\epsilon}=x_{\sigma_0(1)}>0
\end{eqnarray*}
i.e., the maximal weight in (\ref{wl-}) is a positive constant independent of $\epsilon$.

Suppose that Assumption~\ref{ap428} holds. Moreover, assume that
 \begin{eqnarray*}
\liminf_{\epsilon\rightarrow 0} \epsilon\log\left(\min_{x_{i,\epsilon}>x_{j,\epsilon}}\frac{x_{i,\epsilon}}{x_{j,\epsilon}}\right)\geq \alpha>0,
\end{eqnarray*}
where $\alpha$ is a sufficiently large positive constant independent of $\epsilon$.$\hfill\Box$
\end{assumption}

See Figure~\ref{fig:ryepw} for an example of a dimer configuration on a rail yard graph with the left boundary configuration satisfying by a piecewise boundary condition and the right boundary configuration given by an empty partition.

\begin{figure}
\includegraphics[width=1.05\textwidth]{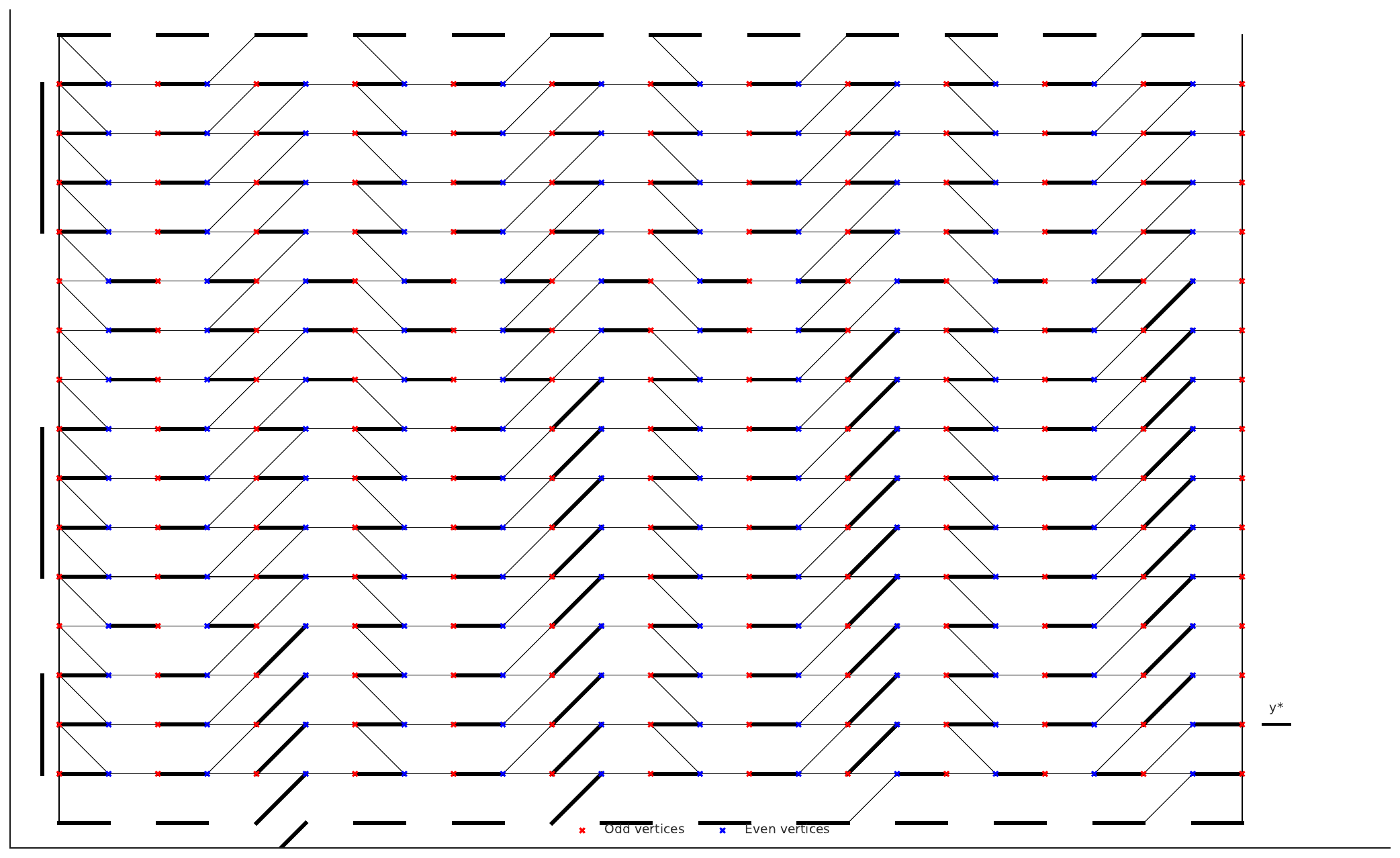}
\caption{A rail yard graph with LR sequence $\underline{a}=\{L,R,L\}$and sign sequence $\underline{b}=\{+,+,-\}$. Odd vertices are represented by red points, and even vertices are represented by blue points. Dark lines represent a dimer covering with left boundary configuration given by a piecewise boundary condition and right boundary configuration given by an empty boundary condition.}\label{fig:ryepw}
\end{figure}

For $N_{L,-}^{(\epsilon)}\geq 1$, let $\lambda^{(l)}\in \YY_{N_{L,-}^{(\epsilon)}}$ be the left boundary partition satisfying Assumption~\ref{ap428}.  Let
\begin{eqnarray*}
\Omega=\left(\Omega_1<\Omega_2<\ldots<\Omega_{N_{L,-}^{(\epsilon)}}\right)=\left(\lambda_{N_{L,-}^{(\epsilon)}}^{(l)}-N_{L,-}^{(\epsilon)},\lambda_{{N_{L,-}^{(\epsilon)}}-1}^{(l)}-N_{L,-}^{(\epsilon)}+1,\ldots,\lambda_1^{(l)}-1\right)
\end{eqnarray*}
Indeed, $\Omega_1,\ldots,\Omega_{N_{L,-}^{(\epsilon)}}$ are the  locations of the $N_{L,-}^{(\epsilon)}$ topmost particles on the left boundary of the rail yard graph, assuming that the vertical coordinate of the $N_{L,-}^{(\epsilon)}$ highest particles is $-N_{L,-}^{(\epsilon)}$ when the left boundary condition is given by the empty partition.
Under Assumption~\ref{ap428}, we may assume
\begin{eqnarray}
\Omega&=&(A_1,A_1+1,\ldots, B_1-1,B_1,\label{abt}\\
&&A_2,A_2+1,\ldots,B_2-1,B_2,\ldots,A_s,A_s+1,\ldots,B_s-1,B_s).\notag
\end{eqnarray}
where 
\begin{eqnarray*}
B_i-A_i+1=K_i.
\end{eqnarray*}
and
\begin{eqnarray*}
\sum_{i=1}^{s}(B_i-A_i+1)=N_{L,-}^{(\epsilon)}.
\end{eqnarray*}
Suppose as $\epsilon\rightarrow 0$,
\begin{eqnarray}
A_i=a_iN_{L,-}^{(\epsilon)}+o(N_{L,-}^{(\epsilon)}),\qquad B_i=b_iN_{L,-}^{(\epsilon)}+o(N_{L,-}^{(\epsilon)}),\qquad \mathrm{for}\ i\in[s],\label{abi}
\end{eqnarray}
and $a_1<b_1<\ldots<a_s<b_s$ are fixed parameters independent of $\epsilon$ and satisfy $\sum_{i=1}^{s}(b_i-a_i)=1$. Under Assumption \ref{ap428}, it is straightforward to check that for $i\in[s]$
\begin{eqnarray*}
b_i&=&\lim_{N\rightarrow\infty}\frac{\mu_{s-i+1}+\sum_{t=1}^{i}K_t}{N_{L,-}^{(\epsilon)}}-1\\
a_i&=&\lim_{N\rightarrow\infty}\frac{\mu_{s-i+1}+\sum_{t=1}^{i-1}K_t}{N_{L,-}^{(\epsilon)}}-1
\end{eqnarray*}

\begin{assumption}\label{ap36}For each $i\in[n]$, write
\begin{align*}
N_{L,-,i}^{(\epsilon)}&=|\{j\in [l^{(\epsilon)}..r^{(\epsilon)}]:a_j=L,b_j=-,\\
&x_{j,\epsilon}\ \mathrm{is\ the}\ i\mathrm{th}\ \mathrm{greatest\ element\ in}\ \{x_{\si_0(1),\epsilon},x_{\si_0(2),\epsilon},\ldots,x_{\si_0(N_{L,-}^{(\epsilon)}),\epsilon}\}\}|
\end{align*}
Suppose that for each $i\in[n]$,
\begin{eqnarray*}
\lim_{\epsilon\rightarrow 0}\frac{N_{L,-,i}^{(\epsilon)}}{N_{L,-}^{(\epsilon)}}=\theta_i\in (0,1);\qquad\mathrm{and}\qquad 
\lim_{\epsilon\rightarrow 0}\frac{N_{L,-}^{(\epsilon)}}{N^{(\epsilon)}}=\rho\in (0,1];
\end{eqnarray*}
where $N^{(\epsilon)}$ is defined by (\ref{dne}); and $\sum_{i\in [n]}\theta_i=1$.$\hfill\Box$
\end{assumption}

\section{Factorization of Schur Generating Functions}
\label{sec:gue}

In this section, we prove the following results.
\begin{enumerate}
\item Under Assumptions \ref{ap428} and \ref{ap32} regarding edge weights, as $\epsilon\rightarrow 0$, we approximately express the Schur generating function with respect to the distribution of partitions near the right boundary into the product of Schur generating functions, each of which depend on one of the distinct weights in 
\begin{align*}
\{x_{j,\epsilon}:j\in [l^{(\epsilon)}..r^{(\epsilon)}]:a_j=L,b_j=-\}
\end{align*}
Let $\mathcal{S}$ denote the set of Schur generating functions that occur as factors of the original generating function. More precisely, $\mathcal{S}$ consists of the Schur generating functions appearing on the right-hand sides of \eqref{e461} and \eqref{e462}.

\item We construct the rail yard graphs with modified edge weights and left boundary conditions such that the distribution of the partitions corresponding to dimer configurations near its right boundary has a Schur generating function as a factor in $\mathcal{S}$.

\end{enumerate}

\subsection{Product of Schur generating functions}

For each $\sigma\in \pi\Sigma_{L,-,\epsilon}$ as defined by (\ref{dps}) and $j\in [N_{L,-}^{(\epsilon)}]$, let
\begin{eqnarray}
\eta_j^{\sigma}(N_{L,-}^{(\epsilon)})=|\{k\in [N_{L,-}^{(\epsilon)}]:k>j,x_{\sigma(k),\epsilon}\neq x_{\sigma(j),\epsilon}\}|.\label{et}
\end{eqnarray}
For $1\leq i\leq n$, let
\begin{align}
\Phi^{(i,\sigma)}(N_{L,-}^{(\epsilon)})=\{\lambda_j^{(l)}+\eta_j^{\sigma}:x_{\sigma(j)}\ \mathrm{is\ the\ }i\mathrm{th\ greatest\ element\ in}\ \{x_{\si_0(1)},x_{\si_0(2)},\ldots,x_{\si_0(N_{L,-}^{(\epsilon)}),\epsilon}\}\}
\label{pis}
\end{align}
and let $\phi^{(i,\sigma)}(N_{L,-}^{(\epsilon)})$ (sometimes we just write $\phi^{(i,\sigma)}$) be the partition with length $|\{1\leq j\leq N: x_j=x_i\}|$ obtained by decreasingly ordering all the elements in $\Phi^{(i,\sigma)}(N_{L,-}^{(\epsilon)})$.

\begin{proposition}\label{p436}Suppose Assumptions \ref{ap32}, \ref{ap428} and \ref{ap36} hold, and let $\alpha$ be given as in Assumption \ref{ap32}. Let 
$\ol{\si}_0$ be the equivalence class of $\si_0$ in $[\Sigma/\Sigma_N^X]^r$; see Section~\ref{sect:A}.
For each given $\{a_i,b_i\}_{i=1}^{n}$, when $\alpha$ is sufficiently large, for any $\sigma\notin \ol{\si}_0$ we have
\begin{eqnarray}
\label{gep}&&\left|\frac{\left(\prod_{i=1}^{n}x_{i,\epsilon}^{|\phi^{(i,\sigma_0)}(N_{L,-}^{(\epsilon)})|}\right)\left(\prod_{i=1}^{n}s_{\phi^{(i,\sigma_0)}(N_{L,-}^{(\epsilon)})}(1,\ldots,1)\right)}{\left(\prod_{i=1}^{n}x_{i,\epsilon}^{|\phi^{(i,\sigma)}(N_{L,-}^{(\epsilon)})|}\right)\left(\prod_{i=1}^{n}s_{\phi^{(i,\sigma)}(N_{L,-}^{(\epsilon)})}(1,\ldots,1)\right)
}\right|\\
&&\times\left|\frac{\left(\prod_{i<j,x_{\sigma_0(i),\epsilon}\neq x_{\sigma_0(j),\epsilon}}\frac{1}{x_{\sigma_0(i),\epsilon}-x_{\sigma_0(j),\epsilon}}\right)}{\left(\prod_{i<j,x_{\sigma(i),\epsilon}\neq x_{\sigma(j),\epsilon}}\frac{1}{x_{\sigma(i),\epsilon}-x_{\sigma(j),\epsilon}}\right)}\right|\notag\geq e^{C\epsilon^{-2}}\notag
\end{eqnarray}
where $C>0$ is a constant independent of $\epsilon$ and $\si$, and increases as $\alpha$ increases. Indeed, we have
\begin{eqnarray*}
\lim_{\alpha\rightarrow\infty} C=\infty.
\end{eqnarray*}
\end{proposition}

\begin{proof}Same arguments as the proof of Proposition 4.5 in \cite{ZL18}.
\end{proof}

\begin{remark}In Proposition~\ref{p437}, we express the Schur polynomial as a sum of products over equivalence classes in $[\Sigma/\Sigma_N^X]^r$. Proposition~\ref{p436} shows that, under suitable assumptions, a single term—the one corresponding to $\overline{\sigma}_0$—dominates, while all other terms are negligible. Consequently, we obtain the asymptotic product formula (\ref{pf}) for the Schur polynomial, which we then exploit to streamline the subsequent analysis.

The limit $\lim_{\alpha\to\infty} C=\infty$ indicates that, as the exponent $\alpha$ increases, imposing that a single edge weight grow exponentially relative to the others amplifies the dominance of one term in the sum \eqref{sws}; see Assumption~\ref{ap32} for the definition of $\alpha$.
\end{remark}

Recall that the Schur generating function was defined in Definition \ref{df21}. By Lemma \ref{lm212}, we obtain
\begin{align}\label{sbf}
\mathcal{S}_{\rho^t,\mathbf{x}^{(L,-,> t)}}(\mathbf{u}^{(L,-,> t)})=
\frac{s_{\lambda^{(l)}}\left(\mathbf{u}^{(L,-,> t)},\mathbf{x}^{(L,-,\leq t)}\right)}{s_{\lambda^{(l)}}(\mathbf{x}^{(L,-)})}
\prod_{i\in[l..t],b_i=+;}
\prod_{j\in[t+1..r],a_j=L,b_j=-.}\frac{\xi_{ij}}{z_{ij}}.
\end{align}

When the edge weights satisfy Assumption \ref{ap32} we have
\begin{align*}
\prod_{i\in[l..t],b_i=+;}
\prod_{j\in[t+1..r],a_j=L,b_j=-.}\frac{\xi_{ij}}{z_{ij}}=\left(1+O(e^{-C\epsilon^{-1}})\right)\prod_{i\in[l..t],b_i=+;}\prod_{j\in[t+1..r],a_j=L,b_j=-;x_{j,\epsilon}=x_{\si_0(1)},\epsilon}\frac{\xi_{ij}}{z_{ij}}.
\end{align*}
By Proposition \ref{p437} Corollary \ref{p438} and Proposition \ref{p436}.
\begin{align}
&\frac{s_{\lambda^{(l)}}\left(\mathbf{u}^{(L,-,> t)},\mathbf{x}^{(L,-,\leq t)}\right)}{s_{\lambda^{(l)}}(\mathbf{x}^{(L,-)})}\label{pf}\\
&=\left(1+e^{-C\epsilon^{-2}}\right)\prod_{i=1}^{n}
\frac{s_{\phi^{(i,\sigma_0)}}\left(\underline{u}^{(i)},1,\ldots,1\right)}{s_{\phi^{(i,\sigma_0)}}\left(1,\ldots,1\right)}\prod_{1\leq i<j\leq N_{L,-}^{(\epsilon)},x_{\si_0(i),\epsilon}\neq x_{\si_0(j),\epsilon}}\frac{x_{\si_0(i),\epsilon}-x_{\si_0(j),\epsilon}}{w_{\si_0(i),\epsilon}-w_{\si_0(j)},\epsilon}\notag,
\end{align}
where
\begin{align*}
\underline{u}^{(i)}=&\left\{\frac{u_j}{x_{j,\epsilon}}:x_{j,\epsilon}\ \mathrm{is\ the\ }i\mathrm{th\ greatest\ element\ in\ } \{x_{\si_0(1),\epsilon},x_{\si_0(1),\epsilon},\ldots,x_{\si_0(N_{L,-}^{(\epsilon),\epsilon}),\epsilon}\},\right.\\
&\left.a_j=L,b_j=-,t+1\leq j\leq r\right\}
\end{align*}

Under Assumption \ref{ap32} we have
\begin{align}
&\prod_{i<j,a_i=a_j=L,b_i=b_j=-,x_{\si_0(i),\epsilon}\neq x_{\si_0(j),\epsilon}}\frac{x_{\si_0(i),\epsilon}-x_{\si_0(j),\epsilon}}{w_{\si_0(i),\epsilon}-w_{\si_0(j),\epsilon}}\label{pf1}\\
&=\left(1+e^{-C{\epsilon^{-1}}}\right)\prod_{i\in [N_{L,-}^{(\epsilon)}]:\sigma_0(i)\in[t+1..r]}\left(\frac{x_{\si_0(i),\epsilon}}{u_{\si_0(i)}}\right)^{\left|\{j\in [N_{L,-}^{(\epsilon)}]:j>i;x_{\si_0(i)}\neq x_{\si_0(j)}\}\right|}\notag
\end{align}

Recall the following lemma proved in \cite{zl202}.
\begin{lemma}\label{l212}
\begin{enumerate}
\item Assume that $x_1,\ldots,x_N\in \CC$ such that
\begin{eqnarray}
|\{i\in [N]: x_i=0\}|=b\in [N]\label{bz}
\end{eqnarray}
Let $\lambda\in \YY_N$ such that there are exactly $a \in[N]$ components of $\lambda$ taking value 0. If $a<b$
 Then
\begin{eqnarray}
s_{\lambda}(x_1,\ldots,x_N)=0.\label{shz}
\end{eqnarray}
\item Let $N$ be a positive integer. Let $b\in[N]$. Assume that 
\begin{eqnarray*}
(u_1,\ldots,u_{N-b},0^{b})\in \CC^N
\end{eqnarray*}
where 
\begin{eqnarray*}
0^{b}=(0,\ldots,0)\in \RR^{b}
\end{eqnarray*}

Let $\tilde{\lambda}\in \YY_N$ such that there are exactly $a \in[N]$ components of $\tilde{\lambda}$ taking value 0.
If $a\geq b$, let 
\begin{eqnarray}
\lambda:&=&(\tilde{\lambda}_1,\tilde{\lambda}_2,\ldots,\tilde{\lambda}_{N-b})\in \GT_{N-b}\label{dtl}
\end{eqnarray}
 Then
\begin{eqnarray*}
s_{\tilde{\lambda}}(u_1,\ldots,u_{N-b},0^b)=s_{\lambda}(u_1,\ldots,u_{N-b}).
\end{eqnarray*}
\end{enumerate}
\end{lemma}
\begin{proof}See Lemmas 3.1 and 3.2 in \cite{zl202}.
\end{proof}

For $i\in[n]$, let
\begin{align*}
\Lambda^{(i,\si)}:=\{\lambda_j^{(l)}:x_{\si(j),\epsilon}\ \mathrm{is\ the\ }i\mathrm{th\ greatest\ element\ in\ }\{x_{\si_0(1),\epsilon},x_{\si_0(2),\epsilon},\ldots,x_{\si_0(N_{L,-}^{(\epsilon)}),\epsilon}\}\},
\end{align*}
and let $\lambda^{(i,\sigma)}$ be the partition formed by elements in $\Lambda^{(i,\si)}$.

Let
\begin{align*}
\mathbf{u}^{(L,-,>t,i)}=&\left\{u_j:x_{j,\epsilon}\ \mathrm{is\ the\ }i\mathrm{th\ greatest\ element\ in\ } \{x_{\si_0(1),\epsilon},x_{\si_0(1),\epsilon},\ldots,x_{\si_0(N_{L,-}^{(\epsilon),\epsilon}),\epsilon}\},\right.\\
&\left.a_j=L,b_j=-,t+1\leq j\leq r\right\}\\
\mathbf{x}^{(L,-,\leq t,i)}=&\left\{x_{j,\epsilon}:x_{j,\epsilon}\ \mathrm{is\ the\ }i\mathrm{th\ greatest\ element\ in\ } \{x_{\si_0(1),\epsilon},x_{\si_0(1),\epsilon},\ldots,x_{\si_0(N_{L,-}^{(\epsilon),\epsilon}),\epsilon}\},\right.\\
&\left.a_j=L,b_j=-,l\leq j\leq t\right\}\\
\mathbf{x}^{(L,-,i)}=&\left\{x_{j,\epsilon}:x_{j,\epsilon}\ \mathrm{is\ the\ }i\mathrm{th\ greatest\ element\ in\ } \{x_{\si_0(1),\epsilon},x_{\si_0(1),\epsilon},\ldots,x_{\si_0(N_{L,-}^{(\epsilon),\epsilon}),\epsilon}\},\right.\\
&\left.a_j=L,b_j=-,l\leq j\leq r\right\}
\end{align*}

By (\ref{pf}),(\ref{pf1}), we obtain
\begin{align*}
\frac{s_{\lambda^{(l)}}\left(\mathbf{u}^{(L,-,> t)},\mathbf{x}^{(L,-,\leq t)}\right)}{s_{\lambda^{(l)}}(\mathbf{x}^{(L,-)})}
=&\left(1+e^{-C\epsilon^{-1}}\right)
\prod_{i=1}^{n}
\frac{s_{\phi^{(i,\sigma_0)}}\left(\mathbf{u}^{(L,-,>t,i)},\mathbf{x}^{(L,-,\leq t,i)}\right)}{s_{\phi^{(i,\sigma_0)}}\left(\mathbf{x}^{(L,-,i)}\right)}\\
&\times\prod_{j\in[N_{L,-}^{(\epsilon)}]:\sigma_0(j)\in[t+1..r]}\left(\frac{x_{\si_0(j),\epsilon}}{u_{\si_0(j)}}\right)^{\eta_j^{\si_0}(N_{L,-}^{(\epsilon)})}\notag,\\
=&\left(1+e^{-C\epsilon^{-1}}\right)
\prod_{i=1}^{n}
\frac{s_{\lambda^{(i,\sigma_0)}}\left(\mathbf{u}^{(L,-,>t,i)},\mathbf{x}^{(L,-,\leq t,i)}\right)}{s_{\lambda^{(i,\sigma_0)}}\left(\mathbf{x}^{(L,-,i)}\right)}
\end{align*}

 Then we can compute the Schur generating function as follows
\begin{align}
&\mathcal{S}_{\rho^t,\mathbf{x}^{(L,-,> t)}}(\mathbf{u}^{(L,-,> t)})\label{sg2}\\
&=(1+e^{-C\mathbf{\epsilon^{-1}}})
\prod_{i=1}^{n}
\frac{s_{\lambda^{(i,\sigma_0)}}\left(\mathbf{u}^{(L,-,>t,i)},\mathbf{x}^{(L,-,\leq t,i)}\right)}{s_{\lambda^{(i,\sigma_0)}}\left(\mathbf{x}^{(L,-,i)}\right)}
\prod_{i\in[l..t],b_i=+;}\prod_{j\in[t+1..r],a_j=L,b_j=-;x_{j,\epsilon}=x_{\si_0(1)},\epsilon}\frac{\xi_{ij}}{z_{ij}}.\notag
\end{align}

The equation (\ref{sg2}) expresses the leading term of the Schur generating function as a product of factors indexed by $i$ $(1\le i\le n)$. In Lemma \ref{l46} below, we show that each $i$-indexed factor can be written as the Schur generating function with respect to the probability measure on dimer coverings with modified left boundary conditions and edge weights.

\begin{lemma}\label{l46}
\begin{align}
&\frac{s_{\lambda^{(1,\sigma)}}\left(\mathbf{u}^{(L,-,>t,1)},\mathbf{x}^{(L,-,\leq t,1)}\right)}{s_{\lambda^{(1,\sigma)}}\left(\mathbf{x}^{(L,-,1)}\right)}
\prod_{i\in[l..t],b_i=+;}\prod_{j\in[t+1..r],a_j=L,b_j=-;x_{j,\epsilon}=x_{\si_0(1)},\epsilon}\frac{\xi_{ij}}{z_{ij}}\label{e461}\\
&=\sum_{\mu\in \YY_{N_{L,-}^{(>t,\epsilon)}}}\tilde{\rho}^t_{\lambda^{(1,\si)}}(\mu)
\frac{s_{\mu}(\underline{u}^{(1)},0^{N_{L,-}^{(>t,\epsilon)}-N_{L,-,1}^{(>t,\epsilon)}})}{s_{\mu}(1,\ldots,1,0^{N_{L,-}^{(>t,\epsilon)}-N_{L,-,1}^{(>t,\epsilon)}})}\notag
\end{align}

For each $i\in[2..n]$,
\begin{align}
&\frac{s_{\lambda^{(i,\sigma)}}\left(\mathbf{u}^{(L,-,>t,i)},\mathbf{x}^{(L,-,\leq t,i)}\right)}{s_{\lambda^{(i,\sigma)}}\left(\mathbf{x}^{(L,-,i)}\right)}=\sum_{\mu\in\YY_{N_{L,-}^{(>t,\epsilon)}}}\tilde{\rho}^{t}_{\lambda^{(i,\si)}}(\mu)\frac{s_{\mu}(\underline{u}^{(i)},0^{N_{L,-}^{(>t,\epsilon)}-N_{L,-,i}^{(>t,\epsilon)}})}{s_{\mu}(1,\ldots,1,0^{N_{L,-}^{(>t,\epsilon)}-N_{L,-,i}^{(>t,\epsilon)}})}\label{e462}
\end{align}
where  
\begin{enumerate}
\item 
\begin{align*}
N_{L,-}^{(>t,\epsilon)}&=|\{j\in [t+1..r^{(\epsilon)}]:a_j=L,b_j=-\}|\\
N_{L,-,i}^{(>t,\epsilon)}&=|\{j\in [t+1..r^{(\epsilon)}]:a_j=L,b_j=-,\\
&x_{j,\epsilon}\ \mathrm{is\ the}\ i\mathrm{th}\ \mathrm{greatest\ element\ in}\ \{x_{\si(1),\epsilon},x_{\si(2),\epsilon},\ldots,x_{\si(N_{L,-}^{(\epsilon)}),\epsilon}\}\}|
\end{align*}
\item $\tilde{\rho}^{t}_{\lambda^{(1,\si)}}$ is the probability measure for the random partition corresponding to dimer configurations incident to odd vertices with abscissa $2t-1$ on a rail yard graph RYG$(l,r,\underline{a},\underline{b})$, left boundary condition given by 
\begin{align*}
\tilde{\lambda}^{(1,\si)}=(\lambda^{(1,\si)},0,\ldots,0),\ \mathrm{s.t.}\ l(\tilde{\lambda}^{(1,\si)})=N_{L,-}^{(\epsilon
)}
\end{align*}
and right boundary condition given by the empty partition; and the weight $\tilde{x}_{j,\epsilon}^{(1)}$ of diagonal edges incident to the column of even vertices with abscissa $2j$ given by
\begin{enumerate}
\item $x_{j,\epsilon}$, if one of the following two conditions holds
\begin{enumerate}
\item $(a_j,b_j)=(L,-)$, and $x_{j,\epsilon}=x_{\si_0(1),\epsilon}$; or
\item $(a_j,b_j)\neq(L,-)$
\end{enumerate}
\item $0$, otherwise.
\end{enumerate}
\item If $i\in[2..n]$,
$\tilde{\rho}^{t}_{\lambda^{(i,\si)}}$  
is the probability measure for the random partition  corresponding to dimer configurations incident to the  column of odd vertices with abscissa $2t-1$ labeled by $(L,-)$, counting from the right on a rail yard graph RYG$(l,r,(L,L,L,\ldots,L),\underline{b})$, left boundary condition given by 
\begin{align*}
\tilde{\lambda}^{(i,\si)}=(\lambda^{(i,\si)},0,\ldots,0),\ \mathrm{s.t.}\ l(\tilde{\lambda}^{(i,\si)})=N_{L,-}^{(\epsilon
)}
\end{align*}
and right boundary condition given by the empty condition; and the weight $\tilde{x}_{j,\epsilon}^{(i)}$ of diagonal edges incident to the column of even vertices with abscissa $2j$ given by
\begin{enumerate}
\item $x_{j,\epsilon}$, if $a_j=L$, $b_j=-$; and $x_{j,\epsilon}$ is the $i$th greatest element in $\{x_{\si_0(1),\epsilon},x_{\si_0(2),\epsilon},x_{\si_0(N_{L,-}^{(\epsilon)}),\epsilon}\}$
\item $0$, otherwise.
\end{enumerate}
\end{enumerate}
\end{lemma}
\begin{proof}By Lemma \ref{l212}, 
\begin{align*}
\frac{s_{\lambda^{(i,\sigma)}}\left(\mathbf{u}^{(L,-,>t,i)},\mathbf{x}^{(L,-,\leq t,i)}\right)}{s_{\lambda^{(i,\sigma)}}\left(\mathbf{x}^{(L,-,i)}\right)}=\frac{s_{\lambda^{(i,\sigma)}}\left(\underline{u}^{(i)},1,\ldots,1\right)}{s_{\lambda^{(i,\sigma)}}\left(1,\ldots,1\right)}=\frac{s_{\tilde{\lambda}^{(i,\sigma)}}\left(\underline{u}^{(i)},1,\ldots,1,0,\ldots,0\right)}{s_{\tilde{\lambda}^{(i,\sigma)}}\left(1,\ldots,1,0,\ldots,0\right)}
\end{align*}
By Lemma \ref{l13} we obtain
\begin{align*}
\tilde{\rho}^t_{\lambda^{(1,\si)}}(\mu)&=\frac{\langle\tilde{\lambda}^{(1,\sigma)}| \Gamma_{a_lb_l}(\tilde{x}_{l,\epsilon}^{(1)})\cdots \Gamma_{a_{t}b_{t}}(\tilde{x}_{t,\epsilon}^{(1)})|\mu \rangle
\langle\mu| \Gamma_{a_{t+1}b_{t+1}}(\tilde{x}_{t+1,\epsilon}^{(1)})\cdots \Gamma_{a_{r}b_{r}}(\tilde{x}_{r,\epsilon}^{(1)})|\emptyset \rangle}{\langle\tilde{\lambda}^{(1,\sigma)}| \Gamma_{a_lb_l}(\tilde{x}_{l,\epsilon}^{(1)})\Gamma_{a_{l+1}b_{l+1}}(\tilde{x}_{l+1,\epsilon}^{(1)})\cdots \Gamma_{a_rb_r}(\tilde{x}_{r,\epsilon}^{(1)})|\emptyset \rangle}
\end{align*}
By Lemma \ref{l12} and Corollary \ref{lc13}(1), we have 
\begin{align*}
&\langle\tilde{\lambda}^{(1,\sigma)}| \Gamma_{a_lb_l}(\tilde{x}_{l,\epsilon}^{(1)})\Gamma_{a_{l+1}b_{l+1}}(\tilde{x}_{l+1,\epsilon}^{(1)})\cdots \Gamma_{a_rb_r}(\tilde{x}_{r,\epsilon}^{(1)})|\emptyset \rangle\\
&=s_{\tilde{\lambda}^{(1,\si)}}\left(\mathbf{x}^{(L,-,1)},0,\ldots,0\right)\prod_{l\leq i<j\leq r;b_i=+,b_j=-,x_{j,\epsilon}=x_{\sigma_0(1),\epsilon}}z_{i,j}\\
&=s_{\lambda^{(1,\si)}}\left(\mathbf{x}^{(L,-,1)}\right)\prod_{l\leq i<j\leq r;b_i=+,b_j=-,x_{j,\epsilon}=x_{\sigma_0(1),\epsilon}}z_{i,j}
\end{align*}
and
\begin{align*}
&\langle\mu| \Gamma_{a_{t+1}b_{t+1}}(\tilde{x}_{t+1,\epsilon}^{(1)})\cdots \Gamma_{a_{r}b_{r}}(\tilde{x}_{r,\epsilon}^{(1)})|\emptyset \rangle\\
&=s_{\mu}\left(\mathbf{x}^{(L,-,>t,1)},0,\ldots,0\right)\prod_{t+1\leq i<j\leq r;b_i=+,b_j=-,x_{j,\epsilon}=x_{\sigma_0(1),\epsilon}}z_{i,j}.
\end{align*}
Hence we have
\begin{align*}
&\sum_{\mu\in \YY_{N_{L,-}^{(>t,\epsilon)}}}\tilde{\rho}^t_{\lambda^{(1,\si)}}(\mu)
\frac{s_{\mu}(\underline{u}^{(1)},0^{N_{L,-}^{(>t,\epsilon)}-N_{L,-,1}^{(>t,\epsilon)}})}{s_{\mu}(1,\ldots,1,0^{N_{L,-}^{(>t,\epsilon)}-N_{L,-,1}^{(>t,\epsilon)}})}\\
&=\sum_{\mu\in \YY_{N_{L,-}^{(>t,\epsilon)}}}
\frac{s_{\mu}(\mathbf{u}^{(L,-,>t,1)},0^{N_{L,-}^{(>t,\epsilon)}-N_{L,-,1}^{(>t,\epsilon)}})}{s_{\mu}(\mathbf{x}^{(L,-,>t,1)},0^{N_{L,-}^{(>t,\epsilon)}-N_{L,-,1}^{(>t,\epsilon)}})}
\\
&\times\frac{\langle\tilde{\lambda}^{(1,\sigma)}| \Gamma_{a_lb_l}(\tilde{x}_{l,\epsilon}^{(1)})\cdots \Gamma_{a_{t}b_{t}}(\tilde{x}_{t,\epsilon}^{(1)})|\mu \rangle
s_{\mu}\left(\mathbf{x}^{(L,-,>t,1)},0,\ldots,0\right)\prod_{t+1\leq i<j\leq r;b_i=+,b_j=-,x_{j,\epsilon}=x_{\sigma_0(1),\epsilon}}z_{i,j}}{s_{\lambda^{(1,\si)}}\left(\mathbf{x}^{(L,-,1)}\right)\prod_{l\leq i<j\leq r;b_i=+,b_j=-,x_{j,\epsilon}=x_{\sigma_0(1),\epsilon}}z_{i,j}}\\ 
&=\sum_{\mu\in \YY_{N_{L,-}^{(>t,\epsilon)}}}
\\
&\times\frac{\langle\tilde{\lambda}^{(1,\sigma)}| \Gamma_{a_lb_l}(\tilde{x}_{l,\epsilon}^{(1)})\cdots \Gamma_{a_{t}b_{t}}(\tilde{x}_{t,\epsilon}^{(1)})|\mu \rangle
s_{\mu}\left(\mathbf{u}^{(L,-,>t,1)},0,\ldots,0\right)\prod_{t+1\leq i<j\leq r;b_i=+,b_j=-,x_{j,\epsilon}=x_{\sigma_0(1),\epsilon}}z_{i,j}}{s_{\lambda^{(1,\si)}}\left(\mathbf{x}^{(L,-,1)}\right)\prod_{l\leq i<j\leq r;b_i=+,b_j=-,x_{j,\epsilon}=x_{\sigma_0(1),\epsilon}}z_{i,j}}\\
&=\sum_{\mu\in \YY_{N_{L,-}^{(>t,\epsilon)}}}\frac{\langle\tilde{\lambda}^{(1,\sigma)}| \Gamma_{a_lb_l}(\tilde{x}_{l,\epsilon}^{(1)})\cdots \Gamma_{a_{t}b_{t}}(\tilde{x}_{t,\epsilon}^{(1)})|\mu \rangle
\langle\mu| \Gamma_{a_{t+1}b_{t+1}}(\tilde{u}_{t+1,\epsilon}^{(1)})\cdots \Gamma_{a_{r}b_{r}}(\tilde{u}_{r,\epsilon}^{(1)})|\emptyset \rangle}{\langle\tilde{\lambda}^{(1,\sigma)}| \Gamma_{a_lb_l}(\tilde{x}_{l,\epsilon}^{(1)})\Gamma_{a_{l+1}b_{l+1}}(\tilde{x}_{l+1,\epsilon}^{(1)})\cdots \Gamma_{a_rb_r}(\tilde{x}_{r,\epsilon}^{(1)})|\emptyset \rangle}\\
&\prod_{l\leq i<j\leq r;b_i=+,b_j=-,x_{j,\epsilon}=x_{\sigma_0(1),\epsilon}}\frac{z_{i,j}}{\xi_{i,j}}
\end{align*}
where for $i\in[t+1..r^{(\epsilon)}]$
\begin{align*}
\tilde{u}_{i,\epsilon}^{(1)}=\begin{cases}u_{i,\epsilon};&\mathrm{if}\ x_{i,\epsilon}=x_{\si_0(1,),\epsilon}\\0;&\mathrm{otherwise}.\end{cases}
\end{align*}

Then we have
\begin{align*}
&\sum_{\mu\in \YY_{N_{L,-}^{(>t,\epsilon)}}}\tilde{\rho}^t_{\lambda^{(1,\si)}}(\mu)
\frac{s_{\mu}(\underline{u}^{(1)},0^{N_{L,-}^{(>t,\epsilon)}-N_{L,-,1}^{(>t,\epsilon)}})}{s_{\mu}(1,\ldots,1,0^{N_{L,-}^{(>t,\epsilon)}-N_{L,-,1}^{(>t,\epsilon)}})}\\
&=\frac{\langle\tilde{\lambda}^{(1,\sigma)}| \Gamma_{a_lb_l}(\tilde{x}_{l,\epsilon}^{(1)})\cdots \Gamma_{a_tb_t}(\tilde{x}_{t,\epsilon}^{(1)})
\Gamma_{a_{t+1}b_{t+1}}(\tilde{u}_{t+1,\epsilon}^{(1)})\cdots \Gamma_{a_rb_r}(\tilde{u}_{r,\epsilon}^{(1)})
|\emptyset \rangle}{\langle\tilde{\lambda}^{(1,\sigma)}| \Gamma_{a_lb_l}(\tilde{x}_{l,\epsilon}^{(1)})\Gamma_{a_{l+1}b_{l+1}}(\tilde{x}_{l+1,\epsilon}^{(1)})\cdots \Gamma_{a_rb_r}(\tilde{x}_{r,\epsilon}^{(1)})|\emptyset \rangle}\prod_{l\leq i<j\leq r;b_i=+,b_j=-,x_{j,\epsilon}=x_{\sigma_0(1),\epsilon}}\frac{z_{i,j}}{\xi_{i,j}}
\end{align*}
Then (\ref{e461}) follows.

The identity (\ref{e462}) can be proved similarly.
\end{proof}

By (\ref{sg2}), (\ref{e461}) and (\ref{e462}), we have
\begin{align}
&\mathcal{S}_{\rho^t,\mathbf{x}^{(L,-,> t)}}(\mathbf{u}^{(L,-,> t)})\label{sg1}\\
&=(1+e^{-C\mathbf{\epsilon^{-1}}})
\prod_{i=1}^{n}\left(\sum_{\mu\in\YY_{N_{L,-}^{(>t,\epsilon)}}}\tilde{\rho}^{t}_{\lambda^{(i,\si_0)}}(\mu)\frac{s_{\mu}(\underline{u}^{(i)},0^{N_{L,-}^{(>t,\epsilon)}-N_{L,-,i}^{(>t,\epsilon)}})}{s_{\mu}(1,\ldots,1,0^{N_{L,-}^{(>t,\epsilon)}-N_{L,-,i}^{(>t,\epsilon)}})}\right)\notag
\end{align}

By Corollaries \ref{lc13}(1), \ref{p438} and (\ref{gep}), we have
\begin{align}
\label{pt}Z_{\lambda^{(l)},\emptyset}(G;\mathbf{x})
&=(1+e^{-C\mathbf{\epsilon^{-1}}})
\left(s_{\lambda^{(1,\sigma_0)}}\left(\mathbf{x}^{(L,-,1)}\right)\prod_{l\leq i<j\leq r;b_i=+,b_j=-,x_{j,\epsilon}=x_{\si_0}(1)}z_{i,j}\right)\\
&\times\prod_{i=2}^{n}
s_{\lambda^{(i,\sigma_0)}}\left(\mathbf{x}^{(L,-,i)}\right)\notag
\end{align}

In Lemma \ref{l46}, we constructed rail-yard graphs RYG$(l,r,\underline{a},\underline{b})$ such that partition functions of perfect matchings on these subgraphs are 
\begin{equation}
s_{\lambda^{(1,\sigma)}}\left(\mathbf{x}^{(L,-,1)}\right)\prod_{l\leq i<j\leq r;b_i=+,b_j=-,x_{j,\epsilon}=x_{\si_0}(1)}z_{i,j}\label{si1}
\end{equation}
and 
\begin{align}
s_{\lambda^{(i,\sigma)}}\left(\mathbf{x}^{(L,-,i)}\right);\qquad \forall i\in[2..n]\label{si2n}
\end{align}
respectively.

Now we discuss (\ref{si1}); which is equal to
\begin{align}
s_{\tilde{\lambda}^{(1,\sigma)}}\left(\mathbf{x}^{(L,-,1)},0^{N_{L,-}^{(\epsilon)}-N_{L,-,i}^{(\epsilon)}}\right)\prod_{l\leq i<j\leq r;b_i=+,b_j=-,x_{j,\epsilon}=x_{\si_0}(1)}z_{i,j}\label{si11}
\end{align}
Consider the weighted rail yard graph given as in Lemma \ref{l46}(2). Then (\ref{si1}), or equivalently (\ref{si11}), gives the partition function of dimer configurations on such a graph with left boundary condition given by $\tilde{\lambda}^{(1,\si)}$ and right boundary condition given by the empty partition. Let $\mathcal{M}_{1,\si}$ be the set consisting of all the dimer configurations on such a graph with given boundary conditions. For each $M_{1,\si}\in \mathcal{M}_{1,\si}$, let
\begin{align*}
\tilde{\lambda}^{(1,\si)}(=\nu^{(M_{1,\si},l)}),\nu^{(M_{1,\si},l-1)},\ldots,\nu^{(M_{1,\si},r-1)}(=\emptyset)
\end{align*}
be the sequence of partitions corresponding to $M_{1,\si}$. For each $t\in [l..r]$, recall that $l_{+}(\nu^{M_{1,\si},t})$ is the total number of nonzero parts in $\nu^{(M_{1,\si},t)}$. 

Assume $2\leq i\leq n$, then (\ref{si2n}) is equal to
\begin{align}
s_{\tilde{\lambda}^{(i,\sigma)}}\left(\mathbf{x}^{(L,-,1)},0^{N_{L,-}^{(\epsilon)}-N_{L,-,i}^{(\epsilon)}}\right)\label{si12}
\end{align}
Consider the weighted rail yard graph given as in Lemma \ref{l46}(3). Then (\ref{si2n}), or equivalently (\ref{si12}), gives the partition function of dimer configurations on such a graph with left boundary condition given by $\tilde{\lambda}^{(i,\si)}$ and right boundary condition given by the empty partition. Let $\mathcal{M}_{i,\si}$ be the set consisting of all the dimer configurations on such a graph with given boundary conditions. For each $M\in \mathcal{M}_{i,\si}$, let
\begin{align*}
\tilde{\lambda}^{(i,\si)}(=\nu^{(M_{i,\si},l)}),\nu^{(M_{i,\si},l-1)},\ldots,\nu^{(M_{i,\si},r-1)}(=\emptyset)
\end{align*}
be the sequence of partitions corresponding to $M_{i,\si}$. 

Then it is straightforward to check the following lemma
\begin{lemma}\label{le409}Assume $2\leq i\leq n$, we have
\begin{align*}
\left|\nu^{(M_{i,\si},t)}_1\right|\leq \tilde{\lambda}^{(i,\si)}_1\leq \lambda^{(l)}_1
\end{align*}
\end{lemma}

Moreover, we have
\begin{lemma}Assume (\ref{c151}) holds.
\begin{align*}
l_{+}(\nu^{(M_{1,\si},t)})\leq N_{L,-,1}^{(>t,\epsilon)}
\end{align*}
\end{lemma}

\begin{proof}Consider dimer configurations on the graph $RYG(t,r,\underline{a},\underline{b})$ with left boundary condition given by $\nu^{(M_{1,\si},t)}$, right boundary condition given by the empty partition, and edge weights satisfies Lemma \ref{l46}(2). First of all, if the edge weights satisfy $\tilde{x}_{j,\epsilon}^{(1)}=0$ for some $j$, then we must have
\begin{align*}
\nu^{(M_{1,\si},j)}=\nu^{(M_{1,\si},j+1)},\ \forall M\in \mathcal{M}_{1,\si}
\end{align*}

We may reorder the operators $\Gamma_{a_i,b_i}$ for $i\in [t..r]\setminus \{j:\tilde{x}_{j,\epsilon}^{(1)}=0\}$, to make sure that all the operators with $b_i=-$ are before all the operators with $b_i=+$. Since (\ref{c151}) holds, whenever $b_i=-$ we have $a_i=L$. After reordering the operators, the corresponding sequence of partitions satisfy
\begin{align*}
\mu^{(M,t)}\succ \tilde{\nu}^{(M,t+1)}\succ\ldots\succ \tilde{\nu}^{(M,t+N_{L,-,1}^{(>t,\epsilon)})}=\emptyset \subseteq \ldots\subseteq \emptyset
\end{align*}
Therefore the total number of nonzero parts of $\nu^{(M_{1,\si},t)}$ is at most $N_{L,-,1}^{(>t,\epsilon)}$, and the lemma follows.
\end{proof}

\begin{lemma}\label{l412}Suppose that (\ref{c151}) and Assumption \ref{ap14} hold. Assume $c$ is a constant satisfying
\begin{align}
c>\left|\frac{\log 2}{\log [\max_{i\in[l..r],b_i=+,a_i=L} x_{i}x_{\si_0(1)}]}\right|\label{clb}
\end{align}
Then for any $\delta>0$,
\begin{align}
\mathbb{P}\left(\max_{t\in[l..r],M\in \mathcal{M}_{1,\si}}\nu_1^{(M,t)}\geq c\delta [N^{(\epsilon)}]^2+N^{(\epsilon)}+\lambda_1^{(1,\si)}\right)<N^{(\epsilon)}\left(\frac{1}{2}\right)^{\delta N^{(\epsilon)}}\label{cn411}
\end{align}
\end{lemma}

\begin{proof}Let $M\in \mathcal{M}_{1,\si}$ be arbitrary. For each $j\in [l..r-1]$, the case 
\begin{align*}
\nu^{(M,j+1)}_1>\nu^{(M,j)}_1
\end{align*}
can occur only when either $\nu^{(M,j)}\prec \nu^{(M,j+1)}$ or $\nu^{(M,j)}\prec'\nu^{(M,j+1)}$. Moreover, if $\nu^{(M,j)}\prec'\nu^{(M,j+1)}$, we have
\begin{align}
\nu^{(M,j+1)}_1\leq \nu^{(M,j)}_1+1.\label{s1}
\end{align}
Consider all the $M\in \mathcal{M}_{1,\si}$ such that  there exists $t\in[l..r]$ such that 
\begin{align}
\nu_1^{(M,t)}\geq c\delta[N^{(\epsilon)}]^2+N^{(\epsilon)}+\lambda_1^{(1,\si)}\label{css}
\end{align}
and let $\mathcal{Z}_{1,+}$ be the partition function (weighted sum) of all such $M$'s satisfying (\ref{css}). Let $\mathcal{Z}_1$ be the partition function of all dimer configurations in $\mathcal{M}_{1,\si}$. Note that $\mathcal{Z}_1$ is given by (\ref{si1}). Then
\begin{align*}
\mathbb{P}\left(\max_{t\in[l..r],M\in \mathcal{M}_{1,\si}}\nu_1^{(M,t)}\geq c\delta[N^{(\epsilon)}]^2+N^{(\epsilon)}+\lambda_1^{(1,\si)}\right)=\frac{\mathcal{Z}_{1,+}}{\mathcal{Z}_1}
\end{align*}
Note that when $a_i=L$, $b_i=+$, $a_j=L$, $b_j=-$ and $x_{j,\epsilon}=x_{\si_0(1)}$, we have
\begin{align*}
z_{i,j}=\frac{1}{1-x_{i}x_{\si_0(1)}}=\sum_{k=0}^{\infty}[x_ix_{\si_0(1)}]^k
\end{align*}
 Let $A_1$ be the contribution of $\nu_1^{(M,t)}-\lambda_1^{(1,\sigma)}$ from $\prec'$ and let $A_2$ be the contribution of $\nu_1^{(M,t)}-\lambda_1^{(1,\sigma_0)}$ from $\prec$. By (\ref{s1}) we see that
 \begin{align*}
 A_1\leq N^{(\epsilon)}-N_{L,-}^{(\epsilon)};
 \end{align*}

When (\ref{css}) holds, we have
\begin{align*}
A_2\geq c\delta [N^{(\epsilon)}]^2.
\end{align*}
Since $\mathcal{Z}_1$ can be computed by (\ref{si1}), $\mathcal{Z}_{1,+}$ corresponds to monomials in the expansion of $\prod_{a_i=a_j=L, b_i=+, b_j=-, x_{j,\epsilon}=x_{\si_0(1)}}z_{ij}$, in which the total degree of $x_ix_{\si_0(1)}$ for all $i\in[l..r],a_i=L,b_i=+$ is at least $c\delta [N^{(\epsilon)}]^2$. Then there exists an $i$ with $a_i=L$ and $b_i=+$ and the degree of $x_ix_{\si_0(1)}$ is at least $c\delta N^{(\epsilon)}$.

Therefore we have 
\begin{align}
\frac{\mathcal{Z}_{1,+}}{\mathcal{Z}_1}\leq N^{(\epsilon)}\left[\max_{i\in[l..r],b_i=+,a_i=L} x_{i}x_{\si_0(1)}\right]^{c\delta N^{(\epsilon)}}\label{pcl}
\end{align}
Then (\ref{cn411}) follows from (\ref{clb}) and (\ref{pcl}).
\end{proof}

\section{Differential Operators and Marginal Distribution}\label{sect:md}

In this section, we introduce new differential operators, and prove that its action on the Schur generating function gives the marginal distribution of certain parts of the random partitions corresponding to random dimer coverings on rail yard graphs; see Proposition \ref{p43}.

\begin{definition}For $1\leq i\leq n$, let  
\begin{align*}
&V_i(\mathbf{u}^{(L,-,>t,i)}):=\prod_{u_{j},u_r\in \mathbf{u}^{(L,-,>t,i)};j<r}(u_j-u_r)\\
&V_{i,*}(\mathbf{u}^{(L,-,>t)}):=\left[\prod_{u_{j}\notin \mathbf{u}^{(L,-,>t,i)},u_r\in \mathbf{u}^{(L,-,>t,i)};j<r}(u_j-u_r)\right]\left[\prod_{u_{j}\notin \mathbf{u}^{(L,-,>t,i)},u_r\in \mathbf{u}^{(L,-,>t,i)};j>r}(u_r-u_j)\right].
\end{align*}

Introduce the family $\mathcal{D}_{i,k}$ of differential operators acting on
symmetric functions $f$ with variables $\mathbf{u}^{(L,-,>,t,i)}$ as follows:
\begin{equation}
\mathcal{D}_{i,k} f=\frac{1}{V_iV_{i,*}}\left(\sum_{u_j\in\mathbf{u}^{(L,-,<t,i)}}\left(u_j\frac{\partial}{\partial
u_j}\right)^k\right)(V_iV_{i,*}\cdot f)\label{dki}.
\end{equation}
Let 
\begin{align*}
\ol{\mathcal{D}}_{i,k}f=\frac{1}{V_i}\left(\sum_{u_j\in\mathbf{u}^{(L,-,<t,i)}}\left(u_j\frac{\partial}{\partial
u_j}\right)^k\right)(V_i\cdot f)
\end{align*}
Define
\begin{equation*}
V(\mathbf{u}^{(L,-,>t)})=\prod_{u_{j},u_r\in \mathbf{u}^{(L,-,>t)};j<r}(u_j-u_r)
\end{equation*}
and 
\begin{equation}
\mathcal{D}_{k} f=\frac{1}{V}\left(\sum_{u_j\in\mathbf{u}^{(L,-,>t)}}\left(u_j\frac{\partial}{\partial
u_j}\right)^k\right)(V\cdot f)\label{dk}.
\end{equation}
\end{definition}
It is straightforward to check the following lemma
\begin{lemma}\label{le52}
\begin{align*}
\mathcal{D}_k=\sum_{i=1}^n\mathcal{D}_{i,k}.
\end{align*}
\end{lemma}

By the Schur branching rule (\ref{sbf}), Proposition \ref{p437} and Corollary \ref{p438}, we have for each $1\leq h\leq n$
\begin{small}
\begin{align*}
\mathcal{D}_{h,k}\mathcal{S}_{\rho^t,\mathbf{x}^{(L,-,> t)}}(\mathbf{u}^{(L,-,> t)})
&=\mathcal{D}_{h,k} \left[\frac{s_{\lambda^{(l)}}\left(\mathbf{u}^{(L,-,> t)},\mathbf{x}^{(L,-,\leq t)}\right)}{s_{\lambda^{(l)}}(\mathbf{x}^{(L,-)})}\prod_{i\in[l..t],b_i=+;}
\prod_{j\in[t+1..r],a_j=L,b_j=-.}\frac{\xi_{ij}}{z_{ij}}\right]\\
&=g(\ol{\si}_0)+
\sum_{\ol{\si}\in\left[\Sigma_{N_{L,-}^{(\epsilon)}}/\Sigma_{N_{L,-}^{(\epsilon)}}^X\right]^r,\ol{\si}\neq\ol{\si}_0}g(\ol{\si})
\end{align*}
\end{small}
where for each $\ol{\si}\in\left[\Sigma_{N_{L,-}^{(\epsilon)}}/\Sigma_{N_{L,-}^{(\epsilon)}}^X\right]^r$, we have
\begin{small}
\begin{align*}
g(\ol{\si}):&=\frac{1}{s_{\lambda^{(l)}}(\mathbf{x}^{(L,-)})}
\mathcal{D}_{h,k} \left[\left(\prod_{i\in[l..t],b_i=+;}
\prod_{j\in[t+1..r],a_j=L,b_j=-.}\frac{\xi_{ij}}{z_{ij}}\right)\right.\\
\notag&\left[\prod_{i<j,a_i=a_j=L,b_i=b_j=-,x_{\si_0(i),\epsilon}\neq x_{\si_0(j),\epsilon}}\frac{1}{w_{\si_0(i),\epsilon}-w_{\si_0(j),\epsilon}}\right]\left.\left((-1)^{f(\ol{\si})}\left(\prod_{i=1}^{n}s_{\phi^{(i,\sigma)}}(\mathbf{u}^{(L,-,>t,i)},\mathbf{x}^{(L,-,\leq t,i)})\right)
\right)
\right]
\end{align*}
\end{small}
and
\begin{small}
\begin{align*}
(-1)^{f(\ol{\si})}=\frac{\prod_{i<j,a_i=a_j=L,b_i=b_j=-,x_{\si(i),\epsilon}\neq x_{\si(j),\epsilon}}\left(w_{\si(i),\epsilon}-w_{\si(j),\epsilon}\right)}{\prod_{i<j,a_i=a_j=L,b_i=b_j=-,x_{\si_0(i),\epsilon}\neq x_{\si_0(j),\epsilon}}\left(w_{\si_0(i),\epsilon}-w_{\si_0(j),\epsilon}\right)}.
\end{align*}
\end{small}
Then we have
\begin{small}
\begin{align*}
g(\ol{\si})&=\frac{(-1)^{f(\ol{\si})}\left(\prod_{i=[1..n],i\neq h}s_{\phi^{(i,\sigma)}}(\mathbf{u}^{(L,-,>t,i)},\mathbf{x}^{(L,-,\leq t,i)})\right)}{s_{\lambda^{(l)}}(\mathbf{x}^{(L,-)})}\left(\prod_{i\in[l..t],b_i=+;}
\prod_{j\in[t+1..r],a_j=L,b_j=-,x_j\notin \mathbf{x}^{(L,-,>t,h)}}\frac{\xi_{ij}}{z_{ij}}\right)\\
&\times\left[\prod_{i<j,a_i=a_j=L,b_i=b_j=-,x_{\si_0(i),\epsilon}\neq x_{\si_0(j),\epsilon},x_{\si_0(i),\epsilon}\notin \mathbf{x}^{(L,-,>t,h)},x_{\si_0(j),\epsilon}\notin \mathbf{x}^{(L,-,>t,h)}}\frac{1}{w_{\si_0(i),\epsilon}-w_{\si_0(j),\epsilon}}\right]A(\ol{\si})
\end{align*}
\end{small}
where
\begin{small}
\begin{align*}
A(\ol{\si})&:=\mathcal{D}_{h,k} \left[\left(\prod_{i\in[l..t],b_i=+}
\prod_{j\in[t+1..r],a_j=L,b_j=-,x_j\in \mathbf{x}^{(L,-,>t,h)}}\frac{\xi_{ij}}{z_{ij}}\right)\left(s_{\phi^{(h,\sigma)}}(\mathbf{u}^{(L,-,>t,i)},\mathbf{x}^{(L,-,\leq t,i)})
\right)\right.\\
\notag&\left[\prod_{i<j,a_i=a_j=L,b_i=b_j=-,x_{\si_0(i),\epsilon}\neq x_{\si_0(j),\epsilon},[x_{\si_0(i)}\in \mathbf{x}^{(L,-,>t,h)} \ \mathrm{or}\ x_{\si_0(i)}\in \mathbf{x}^{(L,-,>t,h)}]}\frac{1}{w_{\si_0(i),\epsilon}-w_{\si_0(j),\epsilon}}\right]
\end{align*}
\end{small}
Let 
\begin{align*}
\underline{u}_j:=\frac{u_j}{x_j};\qquad\forall\ u\in \mathbf{u}^{(L,-)},
\end{align*}
then $\underline{u}_j$ is in a neighborhood of 1, and we can write
\begin{align*}
\mathcal{D}_{i,k} =\frac{1}{V_iV_{i,*}}\left(\sum_{u_j\in\mathbf{u}^{(L,-,<t,i)}}\left(\underline{u}_j\frac{\partial}{\partial
\underline{u}_j}\right)^k\right)V_iV_{i,*}.
\end{align*}
Assume $i<j$, $a_i=a_j=L$, $b_i=b_j=-$, $x_{\si_0(i),\epsilon}\neq x_{\si_0(j),\epsilon},[\si_0(i)=d \ \mathrm{or}\ \si_0(j)=d]$, $x_d\in \mathbf{x}^{(L,-,>t,h)}$.
Under Assumption \ref{ap32} we have
\begin{align}
\left(\frac{\partial}{\partial\underline{u}_d}\right)^{g}\frac{1}{w_{\si_0(i),\epsilon}-w_{\si_0(j),\epsilon}}=\left(-\frac{x_{d,\epsilon}}{w_{\si_0(i),\epsilon}}\right)^g\frac{1}{w_{\si_0(i),\epsilon}}\left(1+O(e^{-cN})\right).\label{gdd}
\end{align}
Note that under Assumption \ref{ap32}, (\ref{gdd}) is exponentially small when $\si_0(j)=d$ compared to when 
$\si_0(i)=d$.

Assume $u_{j}\notin \mathbf{u}^{(L,-,>t,h)},u_d\in \mathbf{u}^{(L,-,>t,i)};j\neq d$, then
\begin{align}
\frac{1}{u_j-u_d}\left(\underline{u}_d\frac{\partial}{\partial\underline{u}_d}\right)^g (u_j-u_d)=-\frac{x_d}{u_j-u_d}.\label{gdd1}
\end{align}
Under Assumption \ref{ap32}, (\ref{gdd1}) is exponentially small when $u_j>u_d$ compared to when $u_j<u_d$.

Moreover, assume $i\in[l..t],b_i=+$, $j\in[t+1..r],a_d=L,b_d=-$, $x_d\in \mathbf{x}^{(L,-,>t,h)}$, we have

\begin{align*}
\left(\frac{\partial}{\partial\underline{u}_d}\right)^{g}\frac{\xi_{id}}{z_{id}}=
\begin{cases}
 \frac{1}{z_{id}}\left(\frac{x_{i,\epsilon}x_{d,\epsilon}}{1-x_{i,\epsilon}u_{d,\epsilon}}\right)^g\frac{1}{1-x_{i,\epsilon}u_{d\epsilon}}, &\mathrm{if}\ b_i=L;\\
\frac{x_{i,\epsilon}x_{d,\epsilon}}{z_{id}},&\mathrm{if}\ b_i=R\ \mathrm{and}\ g=1;\\
0&\mathrm{if}\ b_i=R\ \mathrm{and}\ g>1.
\end{cases}
\end{align*}

In each case above, we have
\begin{align*}
\left(\frac{\partial}{\partial\underline{u}_d}\right)^{g}\frac{\xi_{id}}{z_{id}}=O(e^{-cN^{(\epsilon)}});\ \mathrm{If}\ x_{d,\epsilon}\neq x_{\xi_0(1),\epsilon}.
\end{align*}
Therefore we have
\begin{small}
\begin{align*}
&A(\ol{\si})=\left(1+O(e^{-cN^{(\epsilon)}})\right)\mathcal{D}_{h,k} \left[\left(\prod_{i\in[l..t],b_i=+}
\prod_{j\in[t+1..r],a_j=L,b_j=-,x_j\in \mathbf{x}^{(L,-,>t,h)}}\frac{\xi_{ij}}{z_{ij}}\right)\left(s_{\phi^{(h,\sigma)}}(\mathbf{u}^{(L,-,>t,i)},\mathbf{x}^{(L,-,\leq t,i)})
\right)\right.\\
\notag&\left[\prod_{i<j,a_i=a_j=L,b_i=b_j=-,x_{\si_0(i),\epsilon}\neq x_{\si_0(j),\epsilon},[x_{\si_0(i)}\in \mathbf{x}^{(L,-,>t,h)} \ \mathrm{or}\ x_{\si_0(i)}\in \mathbf{x}^{(L,-,>t,h)}]}\frac{1}{w_{\si_0(i),\epsilon}}\right]\\
&
=\left(1+O(e^{-cN^{(\epsilon)}})\right)\mathcal{D}_{h,k} \left[\left(\prod_{i\in[l..t],b_i=+}
\prod_{j\in[t+1..r],a_j=L,b_j=-,x_j\in \mathbf{x}^{(L,-,>t,h)}}\frac{\xi_{ij}}{z_{ij}}\right)\left(s_{\lambda^{(h,\sigma)}}(\mathbf{u}^{(L,-,>t,h)},\mathbf{x}^{(L,-,\leq t,h)})
\right)\right]\\
&
=\frac{\left(1+O(e^{-cN^{(\epsilon)}})\right)}{V_{h,*}}\ol{\mathcal{D}}_{h,k} \left[\left(\prod_{i\in[l..t],b_i=+}
\prod_{j\in[t+1..r],a_j=L,b_j=-,x_j\in \mathbf{x}^{(L,-,>t,h)}}\frac{\xi_{ij}}{z_{ij}}\right)\right.\\
&\left.\left(s_{\lambda^{(h,\sigma)}}(\mathbf{u}^{(L,-,>t,h)},\mathbf{x}^{(L,-,\leq t,h)})
\right)V_{h,*}\right]\\
&
=\frac{\left(1+O(e^{-cN^{(\epsilon)}})\right)}{Q_h}\ol{\mathcal{D}}_{h,k} \left[\left(\prod_{i\in[l..t],b_i=+}
\prod_{j\in[t+1..r],a_j=L,b_j=-,x_j\in \mathbf{x}^{(L,-,>t,h)}}\frac{\xi_{ij}}{z_{ij}}\right)\right.\\
&\left.\left(s_{\lambda^{(h,\sigma)}}(\mathbf{u}^{(L,-,>t,h)},\mathbf{x}^{(L,-,\leq t,h)})
\right)Q_h\right],
\end{align*}
\end{small}
where
\begin{small}
\begin{align*}
Q_h=\left[\prod_{x_j\in\mathbf{x}^{(L,-,>t,h)}}\underline{u}_j\right]^{\sum_{i=h+1}^n N_{L,-,i}^{(>t,\epsilon)}}.
\end{align*}
\end{small}
Moreover, if $h>1$,
\begin{small}
\begin{align*}
A(\ol{\si})=\frac{(1+O(e^{-cN^{(\epsilon)}}))}{Q_h}\ol{\mathcal{D}}_{h,k} \left(Q_hs_{\lambda^{(h,\sigma)}}(\mathbf{u}^{(L,-,>t,h)},\mathbf{x}^{(L,-,\leq t,h)})
\right).
\end{align*}
\end{small}

When $h=1$, we construct a rail yard graph as in Lemma \ref{l46}(2) with left boundary condition given by
\begin{align*}
\tilde{\lambda}^{(1,\si)}=(\lambda^{(1,\si)},0,\ldots,0),\ \mathrm{s.t.}\ l(\tilde{\lambda}^{(1,\si)})=N_{L,-}^{(\epsilon)}.
\end{align*}
Let $\tilde{\rho}^{t}_{\lambda^{(1,\si)}}$ be the probability measure for the random partition corresponding to dimer configurations incident to odd vertices with abscissa $2t-1$.

When $h>1$, we construct a rail yard graph as in Lemma \ref{l46}(3) with left boundary condition given by 
\begin{small}
\begin{align}
\tilde{\lambda}^{(h,\si)}=(\lambda^{(h,\si)},0,\ldots,0),\ \mathrm{s.t.}\ l(\tilde{\lambda}^{(h,\si)})=N_{L,-}^{(\epsilon
)}.\label{dlh}
\end{align}
\end{small}
Let $\tilde{\rho}^{t}_{\lambda^{(i,\si_0)}}$  
be the probability measure for the random partition  corresponding to dimer configurations incident to the  column of odd vertices with abscissa $2t-1$.

Following the same arguments as in Lemma \ref{l46} with $\si_0$ replaced by $\si$, we obtain for $h=1$
\begin{small}
\begin{align*}
&A(\ol{\si})=(1+O(e^{-cN^{(\epsilon)}}))
s_{\lambda^{(1,\sigma)}}(\mathbf{x}^{(L,-,1)}))
\\
&\times\mathcal{D}_{1,k} \left[\left(\prod_{i\in[l..t],b_i=+}
\prod_{j\in[t+1..r],a_j=L,b_j=-,x_j\in \mathbf{x}^{(L,-,>t,1)}}\frac{\xi_{ij}}{z_{ij}}\right)\frac{s_{\lambda^{(1,\sigma)}}(\mathbf{u}^{(L,-,>t,1)},\mathbf{x}^{(L,-,\leq t,1)}}{s_{\lambda^{(1,\sigma)}}(\mathbf{x}^{(L,-,1)})}
\right]\\
&=(1+O(e^{-cN^{(\epsilon)}}))
s_{\lambda^{(1,\sigma)}}(\mathbf{x}^{(L,-,1)})
\mathcal{D}_{1,k}\left[\sum_{\mu\in \YY_{N_{L,-}^{(>t,\epsilon)}}}\tilde{\rho}^t_{\lambda^{(1,\si)}}(\mu)
\frac{s_{\mu}(\underline{u}^{(1)},0^{N_{L,-}^{(>t,\epsilon)}-N_{L,-,1}^{(>t,\epsilon)}})}{s_{\mu}(1,\ldots,1,0^{N_{L,-}^{(>t,\epsilon)}-N_{L,-,1}^{(>t,\epsilon)}})}\right]\\
&=(1+O(e^{-cN^{(\epsilon)}}))
s_{\lambda^{(1,\sigma)}}(\mathbf{x}^{(L,-,1)})
\mathcal{D}_{1,k}\left[\sum_{\mu\in \YY_{N_{L,-}^{(>t,\epsilon)}}}\tilde{\rho}^t_{\lambda^{(1,\si)}}(\mu)
\frac{s_{\mu}(\underline{u}^{(1)},0^{N_{L,-}^{(>t,\epsilon)}-N_{L,-,1}^{(>t,\epsilon)}})}{s_{\mu}(1,\ldots,1,0^{N_{L,-}^{(>t,\epsilon)}-N_{L,-,1}^{(>t,\epsilon)}})}\right]\\
&=\frac{(1+O(e^{-cN^{(\epsilon)}}))}{Q_h}
s_{\lambda^{(1,\sigma)}}(\mathbf{x}^{(L,-,1)})
\ol{\mathcal{D}}_{1,k}\left[Q_h\sum_{\mu\in \YY_{N_{L,-}^{(>t,\epsilon)}}}\tilde{\rho}^t_{\lambda^{(1,\si)}}(\mu)
\frac{s_{\mu}(\underline{u}^{(1)},0^{N_{L,-}^{(>t,\epsilon)}-N_{L,-,1}^{(>t,\epsilon)}})}{s_{\mu}(1,\ldots,1,0^{N_{L,-}^{(>t,\epsilon)}-N_{L,-,1}^{(>t,\epsilon)}})}\right].
\end{align*}
\end{small}
By Lemma \ref{l212}, we obtain when $h=1$
\begin{small}
\begin{align*}
A(\ol{\si})=\frac{(1+O(e^{-cN^{(\epsilon)}}))}{Q_1}
s_{\lambda^{(1,\sigma)}}(\mathbf{x}^{(L,-,1)})
\ol{\mathcal{D}}_{1,k}\left[Q_1\sum_{\mu\in \YY_{N_{L,-}^{(>t,\epsilon,1)}}}\tilde{\rho}^t_{\lambda^{(1,\si)}}(\mu)
\frac{s_{\mu}(\underline{u}^{(1)})}{s_{\mu}(1,\ldots,1)}\right].
\end{align*}
\end{small}
Similarly by Lemmas \ref{l212} and \ref{l46}, we obtain that for $h>1$
\begin{small}
\begin{align*}
A(\ol{\si})=\frac{(1+O(e^{-cN^{(\epsilon)}}))}{Q_h}
s_{\lambda^{(h,\sigma)}}(\mathbf{x}^{(L,-,h)})
\ol{\mathcal{D}}_{h,k}\left[Q_h\sum_{\mu\in \YY_{N_{L,-}^{(>t,\epsilon,h)}}}\tilde{\rho}^t_{\lambda^{(h,\si)}}(\mu)
\frac{s_{\mu}(\underline{u}^{(1)})}{s_{\mu}(1,\ldots,1)}\right].
\end{align*}
Note that for all $1\leq h\leq n$, we have
\begin{align*}
&\left.\ol{\mathcal{D}}_{h,k}\left[Q_h\sum_{\mu\in \YY_{N_{L,-}^{(>t,\epsilon,h)}}}\tilde{\rho}^t_{\lambda^{(h,\si)}}(\mu)
\frac{s_{\mu}(\underline{u}^{(h)})}{s_{\mu}(1,\ldots,1)}\right]\right|_{\underline{u}^{(h)}=(1,\ldots,1)}\\
&=\mathbb{E}_{\tilde{\rho}^t_{\lambda^{(h,\sigma)}}}\left[\sum_{i=1}^{N_{L,-}^{(>t,\epsilon,h)}}\left(\mu_i+\sum_{j=h}^{n}N_{L,-}^{(>t,\epsilon,j)}-i\right)^k\right].
\end{align*}
\end{small}

Hence we obtain when $1\leq h\leq n$,
\begin{small}
\begin{align*}
&\left.g(\ol{\si})\right|_{\underline{u}^{(h)}=(1,\ldots,1)}=\frac{(-1)^{f(\ol{\si})}\left(\prod_{i=[1..n],i\neq h}s_{\phi^{(i,\sigma)}}(\mathbf{u}^{(L,-,>t,i)},\mathbf{x}^{(L,-,\leq t,i)})\right)}{Q_h s_{\lambda^{(l)}}(\mathbf{x}^{(L,-)})}\\
&\times\left(\prod_{i\in[l..t],b_i=+;}
\prod_{j\in[t+1..r],a_j=L,b_j=-,x_j\notin \mathbf{x}^{(L,-,>t,h)}}\frac{\xi_{ij}}{z_{ij}}\right)\\
&\times\left[\prod_{i<j,a_i=a_j=L,b_i=b_j=-,x_{\si_0(i),\epsilon}\neq x_{\si_0(j),\epsilon},x_{\si_0(i),\epsilon}\notin \mathbf{x}^{(L,-,>t,h)},x_{\si_0(j),\epsilon}\notin \mathbf{x}^{(L,-,>t,h)}}\frac{1}{w_{\si_0(i),\epsilon}-w_{\si_0(j),\epsilon}}\right]\\
&\times(1+O(e^{-cN^{(\epsilon)}}))
s_{\lambda^{(h,\sigma)}}(\mathbf{x}^{(L,-,h)})
\mathbb{E}_{\tilde{\rho}^t_{\lambda^{(h,\sigma)}}}\left[\sum_{i=1}^{N_{L,-}^{(>t,\epsilon,h)}}\left(\mu_i+\sum_{j=h}^n N_{L,-}^{(>t,\epsilon,j)}-i\right)^k\right].
\end{align*}
\end{small}
By Lemmas \ref{l412} and \ref{le409}, we have
\begin{small}
\begin{align*}
\sum_{i=1}^{N_{L,-}^{(>t,\epsilon,h)}}\left(\mu_i+\sum_{j=h}^n N_{L,-}^{(>t,\epsilon,j)}-i\right)^k=O(|N^{(\epsilon)}|^{2k+1}).
\end{align*}
\end{small}
Together with Proposition \ref{p436}, one sees that for any $\ol{\si}\neq \ol{\si}_0$, we have
$|g(\ol{\si})|$ is exponentially small compared to $|g(\ol{\si}_0)|$. Hence we have
\begin{small}
\begin{align*}
&\left.\mathcal{D}_{h,k}\mathcal{S}_{\rho^t,\mathbf{x}^{(L,-,> t)}}(\mathbf{u}^{(L,-,> t)})\right|_{\underline{u}=(1,\ldots,1)}
=\left(1+O(e^{-cN^{(\epsilon)}})\right)g(\ol{\si}_0)\\
&=(1+O(e^{-cN^{(\epsilon)}}))
\mathbb{E}_{\tilde{\rho}^t_{\lambda^{(h,\sigma_0)}}}\left[\sum_{i=1}^{N_{L,-}^{(>t,\epsilon,h)}}\left(\mu_i+\sum_{j=h}^n N_{L,-}^{(>t,\epsilon,j)}-i\right)^k\right].
\end{align*}
\end{small}
Similarly, one can show that 
\begin{small}
\begin{align*}
&\left.\mathcal{D}_{h,k}^l\mathcal{S}_{\rho^t,\mathbf{x}^{(L,-,> t)}}(\mathbf{u}^{(L,-,> t)})\right|_{\underline{u}=(1,\ldots,1)}
=(1+O(e^{-cN^{(\epsilon)}}))
\mathbb{E}_{\tilde{\rho}^t_{\lambda^{(h,\sigma_0)}}}\left[\sum_{i=1}^{N_{L,-}^{(>t,\epsilon,h)}}\left(\mu_i+\sum_{j=h}^n N_{L,-}^{(>t,\epsilon,h)}-i\right)^k\right]^l.
\end{align*}
\end{small}
and for $h_1,h_2,\ldots,h_g$ distinct with $g<n$, we have
\begin{proposition}\label{p43}
\begin{small}
\begin{align}
\label{ee3}&\left.
\left[\prod_{j=1}^g\mathcal{D}_{h_j,k_j}^{l_j}\right]\mathcal{S}_{\rho^t,\mathbf{x}^{(L,-,> t)}}(\mathbf{u}^{(L,-,> t)})\right|_{\underline{u}
=(1,\ldots,1)}\\&=(1+O(e^{-cN^{(\epsilon)}}))
\prod_{j=1}^{g}\mathbb{E}_{\tilde{\rho}^t_{\lambda^{(h_j,\sigma_0)}}}\left[\sum_{i=1}^{N_{L,-}^{(>t,\epsilon,h_j)}}\left(\mu_i+\sum_{j=h}^n N_{L,-}^{(>t,\epsilon,h_j)}-i\right)^{k_j}\right]^{l_j}.\notag
\end{align}
\end{small}
\end{proposition}

In particular, $\mathbb{E}_{\tilde{\rho}^t_{\lambda^{(h_j,\sigma_0)}}}\left[\sum_{i=1}^{N_{L,-}^{(>t,\epsilon,h_j)}}\left(\mu_i+\sum_{j=h}^n N_{L,-}^{(>t,\epsilon,h_j)}-i\right)^{k_j}\right]^{l_j}$ is the expectation of  $\left[\sum_{i=1}^{N_{L,-}^{(>t,\epsilon,h_j)}}\left(\mu_i+\sum_{j=h}^n N_{L,-}^{(>t,\epsilon,h_j)}-i\right)^{k_j}\right]^{l_j}$ when the random partition
\begin{align*}
\left(\mu_1,\mu_2,\ldots, \mu_{N_{L,-}^{(>t,\epsilon,h_j)}}\right)
\end{align*}
has distribution $\tilde{\rho}^t_{\lambda^{(h_j,\sigma_0)}}$. Here $\tilde{\rho}^t_{\lambda^{(h_j,\sigma_0)}}$ is the probability measure for the random partition corresponding to dimer configurations incident to the column of odd vertices with abscissa $2t-1$ on the rail yard graph as in Lemma \ref{l46}(2)(3) with left boundary condition given by $\lambda^{(h_j,\sigma_0)}$, see (\ref{dlh}); and modified edge weights.

\section{Independence}\label{sect:idp}

In this section, we prove that different parts of random partitions are asymptotically independent.

It is straightforward to check the following
\begin{align*}
\mathcal{D}_k s_{\lambda}(u_1,\ldots,u_N)=\left[\sum_{i=1}^N(\lambda_i+N-i)^k\right]s_{\lambda}(u_1,\ldots,u_N)
\end{align*}
Recall the Schur generating function defined as in \ref{df21}. We have
\begin{align*}
\left|[\mathcal{D}_k]^l\mathcal{S}_{\rho,\mathbf{X}}(u_1,\ldots,u_N)\right|_{(u_1,\dots,u_N)=\mathbf{X}}
=\mathbb{E}_{\rho}\left[\sum_{i=1}^N (\lambda_i+N-i)^k\right]^{l}
\end{align*}
Applying the formula to $\mathcal{S}_{\rho^t,\mathbf{x}^{(L,-,> t)}}(\mathbf{u}^{(L,-,> t)})$, where $\rho^t$ is the distribution of partitions corresponding to the dimer configurations adjacent to the odd vertices in column $(2t-1)$, we obtain
\begin{align}
\mathbb{E}_{\rho^t}\prod_{j-1}^{r}\left[\sum_{i=1}^{N^{(>t,\epsilon)}_{L,-}}\left[\lambda_i+N_{L,-}^{(>t,\epsilon)}-i\right]^{k_j}\right]^{l_j}=
\left.\left(\prod_{j=1}^{r}[\mathcal{D}_{k_j}]^{l_j}\right)\mathcal{S}_{\rho^t,\mathbf{x}^{(L,-,> t)}}(\mathbf{u}^{(L,-,> t)})\right|_{\mathbf{u}^{(L,-,> t)}=\mathbf{x}^{(L,-,> t)}}.\label{ee1}
\end{align}
where $N^{(>t,\epsilon)}_{L,-}$ is given as in Lemma \ref{l46}(1). Recall that $n$ is the total number of distinct elements in $\mathbf{x}^{L,-}$. For $1\leq s\leq n$, let
\begin{align*}
I_s^{(>t)}=\sum_{d=1}^{s}N_{L,-,d}^{(>t,\epsilon)};
\end{align*}
and define
\begin{align*}
A_{s,k}^{(>t)}=\sum_{i=I_{s-1}^{(>t)}+1}^{I_s^{(>t)}}\left[\lambda_i+N_{L,-}^{(>t,\epsilon)}-i\right]^k.
\end{align*}
Then we obtain
\begin{small}
\begin{align}
\mathbb{E}_{\rho^t}\prod_{j=1}^{r}\left[\sum_{i=1}^{N^{(>t,\epsilon)}_{L,-}}\left[\lambda_i+N_{L,-}^{(>t,\epsilon)}-i\right]^{k_j}\right]^{l_j}=
\mathbb{E}_{\rho^t}\prod_{j=1}^{r}\left[\sum_{s=1}^{n}A_{s,k_j}^{(>t)}\right]^{l_j}.\label{ee2}
\end{align}
\end{small}

It is straightforward to verify the following lemma:
\begin{proposition}\label{le61}Consider perfect matchings on a rail yard graph RYG$(l,r,\underline{a},\underline{b})$ satisfying Assumptions \ref{ap41}, \ref{ap428}, \ref{ap32}. Let
\begin{align*}
\left(\lambda_1,\ldots,\lambda_{N_{L,-}^{(>t,\epsilon)}}\right)
\end{align*}
be the random partition corresponding to dimer configurations adjacent to a column of odd vertices with abscissa $2t-1$.

Conditional on an event with probability at least $1-e^{-cN^{(\epsilon)}}$. 
\begin{enumerate}
\item The random partition
\begin{align*}
((\lambda_i)_{i=1}^{I_1^{(>t)}},0,\ldots,0)\in \YY^{N_{L,-}^{(>t,\epsilon)}}
\end{align*}
has the distribution 
$\tilde{\rho}^t_{\lambda^{(1,\si_0)}}$, where $\tilde{\rho}^t_{\lambda^{(1,\si_0)}}$ is given as in Lemma \ref{l46}(2).
\item 

The random partition
\begin{align*}
((\lambda_i)_{i=I_{s-1}^{(>t)}+1}^{I_s^{(>t)}},0,\ldots,0)\in \YY^{N_{L,-}^{(>t,\epsilon)}}
\end{align*}
has distribution $\tilde{\rho}^t_{\lambda^{(s,\si_0)}}$, where $\tilde{\rho}^t_{\lambda^{(s,\si_0)}}$ is given as in Lemma \ref{l46}(3).
\end{enumerate}
\end{proposition}

\begin{proof}The lemma follows from Corollary \ref{p438} and Proposition \ref{p436} and Lemma \ref{l46}.
\end{proof}

\begin{proposition}\label{p62}The random partitions 
\begin{align*}
(\lambda_i)_{i=1}^{I_1^{(>t)}}, (\lambda_i)_{i=I_{1}^{(>t)}+1}^{I_2^{(>t)}},\ldots, (\lambda_i)_{i=I_{n-1}^{(>t)}+1}^{I_n^{(>t)}}
\end{align*}
are asymptotically independent in the sense that if 
\begin{align*}
\prod_{j=1}^{r}\left[\sum_{s=1}^{n}A_{s,k_j}^{(>t)}\right]^{l_j}=\sum_i \prod_{s=1}^n f_i(A_{s,k_1}^{(>t)},A_{s,k_2}^{(>t)},\ldots,A_{s,k_n}^{(>t)} )
\end{align*}
Then
\begin{small}
\begin{align*}
\mathbb{E}\prod_{j=1}^{r}\left[\sum_{s=1}^{n}A_{s,k_j}^{(>t)}\right]^{l_j}=\left(1+O(e^{-cN^{(\epsilon)}})\right)\left[\sum_i \prod_{s=1}^n \mathbb{E}f_i(A_{s,k_1}^{(>t)},A_{s,k_2}^{(>t)},\ldots,A_{s,k_n}^{(>t)} )\right]
\end{align*}
\end{small}
\end{proposition}

\begin{proof}The proposition follows from (\ref{ee1}), (\ref{ee2}), Lemma (\ref{le52}), (\ref{ee3}) and Lemma \ref{le61}. 
\end{proof}

\section{Convergence to the spectra of GUE matrices}\label{sect:sgue}

The main aim of this section is to prove the following proposition:
\begin{proposition}\label{p71}Consider perfect matchings on a rail yard graph RYG$(l,r,\underline{a},\underline{b})$ satisfying Assumptions \ref{ap41}, \ref{ap428}, \ref{ap32}. Let
\begin{align*}
\left(\lambda_1,\ldots,\lambda_{N_{L,-}^{(>t,\epsilon)}}\right)
\end{align*}
be the random partition corresponding to dimer configurations adjacent to a column of odd vertices with abscissa $2t-1$.

Assume $t$ is fixed and is finite. 
Assume $1\leq h\leq n$. For random partition $(\lambda_i)_{i=I_{h-1}^{(>t)}+1}^{I_h^{(>t)}}$,
and $1\leq i\leq N_{L,-}^{(\epsilon,h)}$, define
\begin{align*}
b_{N_{L,-}^{(\epsilon,h)},i}
=\frac{\frac{\lambda_{i+I_{h-1}^{(>t)}}+N_{L,-}^{(\epsilon,h)}-i}{\sqrt{N_{L,-}^{(\epsilon,h)}}}-\sqrt{N_{L,-}^{(\epsilon,h)}}A_h}{B_h}
\end{align*}
where $A_h$, $B_h$ are given by (\ref{da1}) (\ref{db1}) (\ref{dah}), (\ref{dbh}).
Then the distribution of $\left\{b_{N_{L,-}^{(\epsilon,1)},i}\right\}_{i=1}^{N_{L,-}^{(\epsilon,h)}}$ converges to $\mathbb{GUE}_{N_{L,-}^{(\epsilon,h)}}$
\end{proposition}

Given Proposition \ref{le61}, it suffices to prove the corresponding result for random partitions with distributions $\tilde{\rho}_{\lambda^{(h,\si_0)}}^t$.

It is straightforward to verify the following lemma:
\begin{lemma}\label{cm}Let $\{J_i\}_{i\in[n]}$ be defined as in Assumption \ref{ap428}.
Under Assumptions \ref{ap428} \ref{ap36}, assume that 
\begin{eqnarray*}
J_i=\left\{\begin{array}{cc}\{d_i,d_i+1,\ldots,d_{i+1}-1\}& \mathrm{if}\ 1\leq i\leq n-1\\ \{d_n, d_n+1,\ldots,s\}& \mathrm{if}\ i=n \end{array}\right.
\end{eqnarray*}
where $d_1,\ldots,d_n$ are positive integers satisfying
\begin{align*}
1=d_1<d_2<\ldots <d_n\leq s
\end{align*}
Let
\begin{align*}
d_{n+1}:=s+1.
\end{align*}
For $j\in[s]$, let $a_j,b_j$ be given by (\ref{abi}). 
If $i\in[n]$, for $0\leq k\leq d_{n+1}-d_n-1$, let
\begin{equation}
\beta_{i,k}
=\sum_{j=1}^{s-d_i-k}\frac{a_{j+1}-b_j}{\theta_i}
+\sum_{r=d_i}^{d_i+k-1}\frac{b_r-a_r}{\theta_i}
\label{dbik}
\end{equation}
and let
\begin{equation}
\gamma_{i,k}=\beta_{i,k}+\frac{b_{s-d_i-k+1}-a_{s-d_i-k+1}}{\theta_i}
=\frac{b_{s-d_i-k+1}-a_1}{\theta_i}\label{dcik}
\end{equation}

Then for $i\in[n]$ the counting measures of $\lambda^{(i,\si_0)}$ converge weakly to a limit measure $\bm_i$ as $\epsilon\rightarrow 0$. Moreover,
if $i\in[n]$, for $0\leq k\leq d_{i+1}-d_i-1$, the limit counting measure $\mathbf{m}_i$ is a probability measure on $[\beta_{i,1},\gamma_{i,d_{i+1}-d_i-1}]$ with density given by
\begin{eqnarray*}
\frac{d\mathbf{m}_i}{dy}=\left\{\begin{array}{cc}1,& \mathrm{if}\ \beta_{i,k}< y< \gamma_{i,k};\\0,& \mathrm{if}\ \gamma_{i,k}\leq y\leq \beta_{i,k+1}. \end{array}\right.
\end{eqnarray*}
\end{lemma}

As we have seen in Lemma \ref{cm}, under Assumption \ref{ap428}, as $\epsilon\rightarrow 0$, the counting measures of  $\phi^{(i,\si_0)}(N_{L,-}^{(\epsilon)})$ converges weakly to a limit measure $\bm_i$.

Let
\begin{equation*}
  P_{k}=\mathrm{diag}\left[v_1,\ldots,v_k\right],\quad
  Q_k=\mathrm{diag}\left[b_{k1},\ldots,b_{kk}\right].
\end{equation*}

By Lemma~\ref{hciz}, we have for any $k$:
\begin{equation}
  \label{lg0}
  \int_{U(k)}\exp\left[\mathrm{Tr}\left(P_k U Q_kU^*\right)\right]dU=
  \left[\prod_{1\leq i<j\leq k}\frac{e^{v_i}-e^{v_j}}{v_i-v_j}\right]
  \frac{s_{\lambda^{k}(N)}(e^{v_1},\ldots,e^{v_k})}{s_{\lambda^k(N)}(1,\ldots,1)}
\end{equation}
For $1\leq h\leq n$, let $\tilde{\rho}^t_{\lambda^{(h,\si_0)}}$ be given as in Lemma \ref{l46}; and let $\ol{\rho}^t_{\lambda^{(h,\si_0)}}$ be the restriction of $\tilde{\rho}^t_{\lambda^{(h,\si_0)}}$ to the first $N_{L,-}^{(>t,\epsilon,h)}=k$ parts.
Then
\begin{multline}
\label{lg1}
\sum_{\lambda\in \YY_{N_{L,-}^{(>t,\epsilon,h)}}}\ol{\rho}^t_{\lambda^{(h,\si_0)}}\int_{U(k)}\exp\left[\frac{1}{\sqrt{N}}\mathrm{Tr}\left(P_kUQ_kU^*\right)\right]dU\\
=\prod_{1\leq i<j\leq k}
\left[\sqrt{N}\frac{e^{v_i/\sqrt{N}}-e^{v_j/\sqrt{N}}}{v_i-v_j}\right] \times
\sum_{\lambda\in \YY_{N_{L,-}^{(>t,\epsilon,h)}}}\ol{\rho}^t_{\lambda^{(h,\si_0)}}
\frac{s_{\lambda}(e^{v_1/\sqrt{N}},\dotsc e^{v_k
  /\sqrt{N}})}{s_{\lambda}(1,\dotsc,1)}\\
  =
  \prod_{1\leq i<j\leq k}
\left[\sqrt{N}\frac{e^{v_i/\sqrt{N}}-e^{v_j/\sqrt{N}}}{v_i-v_j}\right] \times
  \mathcal{S}_{\ol{\rho}^t_{\lambda^{(h,\si_0)}}}(e^{v_1/\sqrt{N}},\dotsc,e^{v_k/\sqrt{N}})
\end{multline}
note that the denominator of the last fraction being exactly 
$\mathcal{S}_{\ol{\rho}^t_{\lambda^{(h,\si_0)}}}(e^{v_1/\sqrt{N}},\dotsc,e^{v_k/\sqrt{N}})$
by Definition~\ref{df21}.

We then use Lemma~\ref{l46} to re-express the Schur generating function:
\begin{align*}
\mathcal{S}_{\ol{\rho}^t_{\lambda^{(1,\si_0)}}}(\mathbf{u}^{(L,-,>t,1)})
=
\frac{s_{\lambda^{(1,\sigma)}}\left(\underline{u}^{(L,-,>t,1)},1^{N_{L,-}^{(\leq t,1)}}\right)}{s_{\lambda^{(1,\sigma)}}\left(1^{N_{L,-}^{(\leq t,1)}}\right)}
\prod_{i\in[l..t],b_i=+;}\prod_{j\in[t+1..r],a_j=L,b_j=-;x_{j,\epsilon}=x_{\si_0(1)},\epsilon}\frac{\xi_{ij}}{z_{ij}},
\end{align*}
and for $2\leq h\leq k$
\begin{align*}
\mathcal{S}_{\ol{\rho}^t_{\lambda^{(h,\si_0)}}}(\mathbf{u}^{(L,-,>t,h)})
=
\frac{s_{\lambda^{(h,\sigma_0)}}\left(\underline{u}^{(L,-,>t,h)},1^{N_{L,-}^{(\leq t,h)}}\right)}{s_{\lambda^{(h,\sigma_0)}}\left(1^{N_{L,-}^{(\leq t,h)}}\right)},
\end{align*}

This ratio of Schur functions
appears also when using again Lemma~\ref{hciz} to rewrite the matrix
integral over $U(N_{L,-}^{(h,\si_0)})$ this time.
More precisely, let
\begin{align*}
&R_{N_{L,-}^{(\epsilon,h)}}=\mathrm{diag}\left[v_1,\dotsc,v_k,0,\dotsc,0\right],
\\
&Q_{N_{L,-}^{(\epsilon,h)}}=\mathrm{diag}\left[\lambda^{(h,\si_0)}_1+N_{L,-}^{(\epsilon,h)}-1,\lambda^{(h,\si_0)}_2+N_{L,-}^{(\epsilon,h)}-2\ldots,\lambda^{(h,\si_0)}_{N_{L,-}^{(\epsilon,h)}}\right].
\end{align*}
For $1\leq i\leq k$, let $\underline{u}_i=e^{\frac{v_i}{\sqrt{N_{L,-}^{\epsilon,h}}}}$.
Then,
\begin{small}
\begin{align}
  \label{lg2b}
&\frac{%
s_{\lambda^{(h,\si_0)}}\left(1,\ldots,1,\underline{u}^{(L,-,>t,h)}\right)
}{%
s_{\lambda^{(h,\si_0)}}\left(1,\dotsc,1\right)}
=
\int_{U(N_{L,-}^{(\epsilon,h)})}\exp\left[\frac{1}{\sqrt{N_{L,-}^{(\epsilon,h)}}} \mathrm{Tr}\left(R_{N_{L,-}^{(\epsilon,h)}} U Q_{N_{L,-}^{(\epsilon,h)}} U^*\right)\right]dU \times\\
&\left[\prod_{1\leq i<j\leq k}\!\!\!
\sqrt{N_{L,-}^{(\epsilon,h)}}\frac{e^{\frac{v_i}{\sqrt{N_{L,-}^{(\epsilon,h)}}}}-e^{\frac{v_j}{\sqrt{N_{L,-}^{(\epsilon,h)}}}}}{v_i-v_j}\right]^{-1}\!\!\!
\left[\prod_{\substack{{1\leq i\leq k}\\{1\leq j\leq N_{L,-}^{(\epsilon,h)}-k}}}\!\!\!
  \frac{e^{\frac{v_i}{\sqrt{N_{L,-}^{(\epsilon,h)}}}}-1}{\frac{v_i}{\sqrt{N_{L,-}^{(\epsilon,h)}}}}\right]^{-1}\!\!\!\notag
\end{align}
\end{small}

Plugging \eqref{lg2b}
into~\eqref{lg1}, one gets if $h\geq 2$:
\begin{small}
\begin{align*}
&\log\sum_{\lambda^{(\epsilon,h)}\in \YY_{N_{L,-}^{(\epsilon,h)}}}
\ol{\rho}^t_{N_{L,-}^{(\epsilon,h)}}\int_{U(k)} 
\exp\left[\frac{1}{\sqrt{N_{L,-}^{(\epsilon,h)}}}\mathrm{Tr}\left(P_kUQ_kU^*\right)\right]dU\\
&=\log\int_{U(N_{L,-}^{(h,\si_0)})}\exp\left[\frac{1}{\sqrt{N_{L,-}^{(h,\si_0)}}}\mathrm{Tr}\left(R_{N_{L,-}^{(h,\si_0)}}UQ_{N_{L,-}^{(h,\si_0)}}U^*\right)\right]dU-\sum_{\substack{{1\leq i\leq k}\\{1\leq j\leq N_{L,-}^{(h,\si_0)}-k}}}
\log\left[
  \frac{e^{\frac{v_i}{\sqrt{N_{L,-}^{(h,\si_0)}}}}-1}{\frac{v_i}{\sqrt{N_{L,-}^{(h,\si_0)}}}}
\right]
\\
&-\sum_{1\leq i<j\leq k}
\log\left[
  \sqrt{N_{L,-}^{(h,\si_0)}}\frac{e^{\frac{v_i}{\sqrt{N_{L,-}^{(h,\si_0)}}}}-e^{\frac{v_j}{\sqrt{N_{L,-}^{(h,\si_0)}}}}}{v_i-v_j}
\right].
\end{align*}
\end{small}
If $h=1$,
\begin{small}
\begin{align*}
&\log\sum_{\lambda^{(1,\si_0)}\in \YY_{N_{L,-}^{(1,\si_0)}}}
\ol{\rho}^t_{N_{L,-}^{(1,\si_0)}}\int_{U(k)} 
\exp\left[\frac{1}{\sqrt{N_{L,-}^{(\epsilon,h)}}}\mathrm{Tr}\left(P_kUQ_kU^*\right)\right]dU\\
&=\log\int_{U(N_{L,-}^{(\epsilon,h)})}\exp\left[\frac{1}{\sqrt{N_{L,-}^{(\epsilon,h)}}}\mathrm{Tr}\left(R_{N_{L,-}^{(\epsilon,h)}}UQ_{N_{L,-}^{(\epsilon,h)}}U^*\right)\right]dU-\sum_{\substack{{1\leq i\leq k}\\{1\leq j\leq N_{L,-}^{(\epsilon,h)}-k}}}
\log\left[
  \frac{e^{\frac{v_i}{\sqrt{N_{L,-}^{(\epsilon,h)}}}}-1}{\frac{v_i}{\sqrt{N_{L,-}^{(\epsilon,h)}}}}
\right]
\\
&-\sum_{1\leq i<j\leq k}
\log\left[
  \sqrt{N_{L,-}^{(\epsilon,h)}}\frac{e^{\frac{v_i}{\sqrt{N_{L,-}^{(\epsilon,h)}}}}-e^{\frac{v_j}{\sqrt{N_{L,-}^{(\epsilon,h)}}}}}{v_i-v_j}
\right]+\sum_{i\in[l..t],b_i=+;}\sum_{j\in[t+1..r],a_j=L,b_j=-;x_{j,\epsilon}=x_{\si_0(1)},\epsilon}\log\frac{\xi_{ij}}{z_{ij}}
\end{align*}
\end{small}

Expand the sums using the  two expansions
$\log\frac{e^u-e^v}{u-v}=\frac{u+v}{2}+\frac{(u-v)^2}{24}+O(|u|^3+|v|^3)$, and
$\log(\frac{1+ye^v}{1+y})=\frac{y}{1+y}v+\frac{y}{2(1+y)^2}v^2+O(|v|^3)$ as $u$ and $v$
tend to 0, to get that:
\begin{small}
\begin{align*}
&\sum_{\substack{{1\leq i\leq k}\\{1\leq j\leq N_{L,-}^{(\epsilon,h)}-k}}}
\log\left[
  \frac{%
    e^{v_i/\sqrt{N_{L,-}^{(\epsilon,h)}}}-1
  }{%
    \frac{v_i}{\sqrt{N_{L,-}^{(\epsilon,h)}}}
  }
\right]=\left(N_{L,-}^{(\epsilon,h)}-k\right)\sum_{1\leq i\leq k}\left[\frac{v_i}{2\sqrt{N_{L,-1}^{(\epsilon,h)}}}+\frac{v_i^2}{24N_{L,-}^{(\epsilon,h)}}\right]+O([N_{L,-}^{(\epsilon,h
)}]^{-\frac{1}{2}}),\\
&\sum_{1\leq i<j\leq k}
\log\left[
  \sqrt{N_{L,-}^{(h,\si_0)}}\frac{e^{\frac{v_i}{\sqrt{N_{L,-}^{(h,\si_0)}}}}-e^{\frac{v_j}{\sqrt{N_{L,-}^{(h,\si_0)}}}}}{v_i-v_j}
\right]=O([N_{L,-}^{(\epsilon,h
)}]^{-\frac{1}{2}}).
\end{align*}
Moreover,
\begin{align*}
&\sum_{i\in[l..t],b_i=+;}\sum_{j\in[t+1..r],a_j=L,b_j=-;x_{j,\epsilon}=x_{\si_0(1)},\epsilon}\log\frac{\xi_{ij}}{z_{ij}}\\
=&\sum_{i\in[l..t],a_i=L;b_i=+;}\sum_{j=1}^k\log\frac{1-x_ix_{\si_0(1)}}{1-x_ix_{\si_0(1)}e^{\frac{v_j}{\sqrt{N_{L,-}^{(\epsilon,1)}}}}}+\sum_{i\in[l..t],a_i=L;b_i=+;}\sum_{j=1}^{k}\log\frac{1+x_ix_{\si_0(1)}e^{\frac{v_j}{\sqrt{N_{L,-}^{(\epsilon,1)}}}}}{1+x_ix_{\si_0(1)}}\\
=&\sum_{i\in[l..t],a_i=L;b_i=+;}\sum_{j=1}^k
\left[\frac{x_ix_{\si_0(1)}}{1-x_ix_{\si_0(1)}}\frac{v_j}{\sqrt{N_{L,-}^{(\epsilon,1)}}}
+\frac{x_ix_{\si_0(1)}}{2(1-x_ix_{\si_0(1)})^2}\frac{v_j^2}{N_{L,-}^{(\epsilon,1)}}
+O([N_{L,-}^{(\epsilon,1)}]^{-\frac{3}{2}})\right]\\
+&\sum_{i\in[l..t],a_i=L;b_i=+;}\sum_{j=1}^k
\left[\frac{x_ix_{\si_0(1)}}{1+x_ix_{\si_0(1)}}\frac{v_j}{\sqrt{N_{L,-}^{(\epsilon,1)}}}
+\frac{x_ix_{\si_0(1)}}{2(1+x_ix_{\si_0(1)})^2}\frac{v_j^2}{N_{L,-}^{(\epsilon,1)}}
+O([N_{L,-}^{(\epsilon,1)}]^{-\frac{3}{2}})\right]
\end{align*}
\end{small}

Then we use Lemma~\ref{lem:asympexpiz} below to get the asymptotic behavior of the
integral over $U(N_{L,-}^{(\epsilon,h)})$, to obtain that
\begin{align*}
 & \log\int_{U(N_{L,-}^{(\epsilon,h)})}\exp\left[\frac{1}{\sqrt{N_{L,-}^{(\epsilon,h)}}}\mathrm{Tr}\left(R_{N_{L,-}^{(\epsilon,h)}}UQ_{N_{L,-}^{(\epsilon,h)}}U^*\right)\right]dU\\
 &=\psi_{1,h}
  \left(\sum_{i=1}^k v_i\right)+\frac{\psi_{2,h}-\psi_{1,h}^2}{2}\left(\sum_{i=1}^k
  v_i^2\right)+o(1),
\end{align*}
where $\psi_{1,h}$ and $\psi_{2,h}$ are respectively the first and second moments of the
limiting measure $\mathbf{m}_{h}$ (see Lemma \ref{cm}), as $N_{L,-}^{(\epsilon,h)}$ goes to infinity.
Bringing all the pieces together, and defining
\begin{align}
  A_1 &= \psi_{1,1}-\frac{1}{2}+\sum_{i\in[l..t],a_i=L;b_i=+;}\frac{x_ix_{\si_0(1)}}{1-x_ix_{\si_0(1)}}+\sum_{i\in[l..t],a_i=L;b_i=+;}\frac{x_ix_{\si_0(1)}}{1+x_ix_{\si_0(1)}},\label{da1}
\\
    B_1 &= \psi_{2,1}-\psi_{1,1}^2-\frac{1}{12}+\sum_{i\in[l..t],a_i=L;b_i=+;}\frac{x_ix_{\si_0(1)}}{[1+x_ix_{\si_0(1)}]^2}+\sum_{i\in[l..t],a_i=L;b_i=+;}\frac{x_ix_{\si_0(1)}}{[1-x_ix_{\si_0(1)}]^2},\label{db1}
\end{align}
and for $h\geq 2$, define
\begin{align}
  A_h &= \psi_{1,h}-\frac{1}{2},\label{dah}
\\
    B_h &= \psi_{2,h}-\psi_{1,h}^2,\label{dbh}
\end{align}
one finally obtain that
\begin{equation*}
  \int_{U(k)}\exp\left[\frac{1}{\sqrt{N_{L,-}^{(\epsilon,h)}}}\mathrm{Tr}\left(P_k U Q_kU^*\right)\right]dU=
  \exp\left(\sqrt{N_{L,-}^{(\epsilon,h)}}A_h\sum_i v_i+
    \frac{1}{2}B_h \sum_i v_i^2+o(1)\right).  
\end{equation*}
Therefore, according to Lemma~\ref{lgue}, the distribution of the diagonal
coefficients  of
\begin{equation*}
\frac{1}{\sqrt{N_{L,-}^{(\epsilon,h)}}B_h}(Q_k-\sqrt{N_{L,-}^{(\epsilon,h)}} A_h \operatorname{Id}_k)
\end{equation*}
which are exactly the
$\tilde{b}_{kl}^{(N)}$, converges to $\PP_{\GUE_k}$.

\begin{lemma}
  \label{lem:asympexpiz}
  Let $k\in \NN$ be fixed.
  For any $N$, let $\omega(N)\in \GT_N^+$ and
  \begin{align*}
    q_i^{(N)}&=\omega_i(N)+N-i,\quad
    Q_N=\operatorname{diag}\left[q_1^{(N)},\dotsc, q_N^{N}\right],\\
    P_N&=\mathrm{diag}\left[
    v_1,\ldots,v_k,0,\ldots,0\right].
\end{align*}
Assume that 
\begin{itemize}
\item $\left(\omega(N)\right)_{N\in\NN}$ is a regular sequence of signatures,
\item as $N\rightarrow\infty$, the counting measures $\bm_{\omega(N)}$ converge
  weakly to $\bm_{\omega}$,
\end{itemize}
 Then as $N\rightarrow\infty$, we have the following asymptotic expansion
\begin{equation}
\log \int_{U(N)} \exp\left[\frac{1}{\sqrt{N}}\mathrm{Tr}\left(P_NU Q_N U^*\right)\right]dU
= K_1p_1(v_1,\ldots,v_k) N^{1/2} + \frac{K_2}{2!}p_2(v_1,\ldots,v_k)+o(1)
\label{eq:asymptexpiz}
\end{equation}
where $K_d$ is $(d-1)!$ multiplying the $d$th free cumulant of the measure
$\bm_{\omega}$ (in particular we have $K_1=\psi_1, K_2=\psi_2-\psi_1^2$)
and the $p_d$ is the $d$th power sum:
\begin{equation*}
p_d(v_1,\ldots,v_k)=\sum_{i=1}^{k}v_i^d.
\end{equation*}
\end{lemma}

\begin{proof}
  See \cite[Corollary~6]{JN14}; see also
  \cite{ggn}).
\end{proof}

\appendix

\section{A formula to compute Schur functions}\label{sect:A}

Let $\lambda(N)\in \YY_N$. Let $\Sigma_N$ be the permutation group of $N$ elements and let $\sigma\in \Sigma_N$. Let 
\begin{eqnarray*}
X=(x_1,\ldots,x_N).
\end{eqnarray*}
Assume that there exists a positive integer $n\in[N]$ such that $x_1,...,x_n$ are pairwise distinct and $\{x_1,...,x_n\}=\{x_1,...,x_N\}$. For $j\in[N]$, let
\begin{eqnarray}
\eta_j^{\sigma}(N)=|\{k:k>j,x_{\sigma(k)}\neq x_{\sigma(j)}\}|.\label{et}
\end{eqnarray}
For $1\leq i\leq n$, let
\begin{eqnarray}
\Phi^{(i,\sigma)}(N)=\{\lambda_j(N)+\eta_j^{\sigma}(N):x_{\sigma(j)}=x_i\}\label{pis}
\end{eqnarray}
and let $\phi^{(i,\sigma)}(N)$ be the partition with length $|\{1\leq j\leq N: x_j=x_i\}|$ obtained by decreasingly ordering all the elements in $\Phi^{(i,\sigma)}(N)$. Let $\Sigma_N^{X}$ be the subgroup $\Sigma_N$ that preserves the value of $X$; more precisely
\begin{eqnarray*}
\Sigma_N^{X}=\{\sigma\in \Sigma_N: x_{\sigma(i)}=x_i,\ \mathrm{for}\ i\in[N]\}.
\end{eqnarray*}
Let $[\Sigma/\Sigma_N^X]^r$ be the collection of all the right cosets of $\Sigma_N^X$ in $\Sigma_N$. More precisely,
\begin{eqnarray*}
[\Sigma/\Sigma_N^X]^r=\{\Sigma_N^X\sigma:\sigma\in \Sigma_N\},
\end{eqnarray*}
where for each $\sigma\in \Sigma_N$
\begin{eqnarray*}
\Sigma_N^X\sigma=\{\xi\sigma:\xi\in \Sigma_N^X\}
\end{eqnarray*}
and $\xi\sigma\in \Sigma_N$ is defined by
\begin{eqnarray*}
\xi\sigma(k)=\xi(\sigma(k)),\ \mathrm{for}\ k\in[N].
\end{eqnarray*}

Let $k\in[N]$. Let
\begin{eqnarray}
\label{dw}w_{i}=\left\{\begin{array}{cc}u_i&\mathrm{if}\ 1\leq i\leq k\\x_i&\mathrm{if}\ k+1\leq i\leq N\end{array}\right.
\end{eqnarray}

\begin{proposition}\label{p437}Let $\{w_i\}_{i\in[N]}$ be given by (\ref{dw}). 
Define
\begin{eqnarray*}
\underline{u}^{(i)}=\left\{\frac{u_j}{x_i}:x_j=x_i\right\}
\end{eqnarray*}
Then we have the following formula
\begin{eqnarray}
&&\label{sws}s_{\lambda}(w_1,\ldots,w_N)\\
&=&\sum_{\ol{\sigma}\in[\Sigma_N/\Sigma_N^X]^r} \left(\prod_{i=1}^{n}x_i^{|\phi^{(i,\sigma)}(N)|}\right)\left(\prod_{i=1}^{n}s_{\phi^{(i,\sigma)}(N)}\left(\underline{u}^{(i)},1,\ldots,1\right)\right)\notag\\
&&\times\left(\prod_{i<j,x_{\sigma(i)}\neq x_{\sigma(j)}}\frac{1}{w_{\sigma(i)}-w_{\sigma(j)}}\right)\notag
\end{eqnarray}
where $\sigma\in \ol{\sigma}\cap \Sigma_N$ is a representative. 
\end{proposition}

\begin{proof}See Proposition 4.1 of \cite{Li21}.
\end{proof}

\begin{corollary}\label{p438}
\begin{eqnarray}
&&\label{sws1}s_{\lambda}(x_1,\ldots,x_N)\\
&=&\sum_{\ol{\sigma}\in[\Sigma_N/\Sigma_N^X]^r} \left(\prod_{i=1}^{n}x_i^{|\phi^{(i,\sigma)}(N)|}\right)\left(\prod_{i=1}^{n}s_{\phi^{(i,\sigma)}(N)}\left(1,\ldots,1\right)\right)\left(\prod_{i<j,x_{\sigma(i)}\neq x_{\sigma(j)}}\frac{1}{x_{\sigma(i)}-x_{\sigma(j)}}\right)\notag
\end{eqnarray}
where $\sigma\in \ol{\sigma}\cap \Sigma_N$ is a representative. 
\end{corollary}

\begin{proof}The corollary follows from Proposition \ref{p437} by letting $u_i=x_i$ for all $i\in[k]$.
\end{proof}

\bigskip
\noindent\textbf{Acknowledgements.}\ ZL acknowledges support from National Science Foundation DMS 1608896 and Simons Foundation grant 638143. ZL thanks the anonymous reviewers for their careful reading and for valuable suggestions that improved the clarity and readability of the paper.

\bibliography{sq,fpmm2}
\bibliographystyle{plain}

\end{document}